\documentclass{ amsart}
\usepackage{amsmath}
\usepackage{amsthm}
\usepackage{amscd}
\usepackage{amssymb}
\usepackage{graphicx}
\newtheorem{thm}{Theorem}
\newtheorem{lem}{Lemma}
\newtheorem{cor}{Corollary}
\newtheorem{Def}{Definition}
\newtheorem{Ex}{Example}
\newtheorem{rmk}{Remark}
\begin{document}
\title{The W-Polynomial and the Mahler Measure of the Kauffman Bracket }
\author{ Robert G. Todd\\ University of Nebraska at Omaha} 
\maketitle

\begin{abstract}The W-polynomial is applied in two ways to questions involving the Kauffman bracket of some families of links. First we find a geometric property of a link diagram, which is less than or equal to the twist number, that bounds the Mahler measure of the Kauffman bracket. Second we find a general form for the Kauffman bracket of a link found by surgering in a single rational tangle along n unlinked components, all in a particular annulus. We then give a condition under which the Mahler measure of the Kauffman bracket of such families diverges. We give examples of the condition in action.
\end{abstract}
\section{Introduction}

In \cite{Twisting}, Champanerkar and Kofman explore the Mahler measure of the Jones polynomials of links in a family obtained by $-\frac{1}{m}$ surgery on a diagram, concluding that the Mahler measure converges. In \cite{cyc} they are able to find a bound on the Mahler measure of the Jones polynomials of families of links all with the same twist number. Each of these results has a geometric analogue. First, Thurston's hyperbolic Dehn surgery theorem, the volume of the complement of a hyperbolic knot converges under such Dehn surgeries. Second, Lackenby's result that the twist number is an upper bound on the volume of the complement of a hyperbolic knot. 

Here we wish to investigate some issues regarding Mahler measure with these ideas in mind. Consider the Kauffman bracket of a link as defined by
\newline

1. The Kauffman bracket of n disjoint simple closed curves is $d^{n-1}$ where $d=-A^{-2}-A^{2}$
\newline

2. $<D>=A<D_{+}>+A^{-1}<D_{-}>$, where  $D_{+}$ and $D_{-}$ are the diagrams achieved by a positive and negative smoothing of a crossing v in diagram D, as in Figure \ref{fig: smoothing}

Let $W(D)$ be the writhe of the diagram D. Then the Jones  polynomial of the link L with diagram D  is 
$$V_{L}(t):=(-A^{3})^{-w(D)}<D>|_{A=t^{-\frac{1}{4}}}$$

\begin{figure}[h]
\vspace*{.5in}

\hspace*{-.5in}\includegraphics{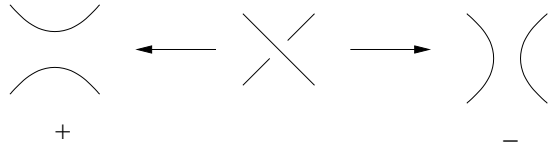}

\caption{Positive and Negative smoothing of a crossing} \label{fig: smoothing}
\end{figure}
	
	In sections 2 and 3 we introduce the W-polynomial and its relation to the Kauffman bracket. In section 4 we review the construction of the twist polynomial from \cite{cyc} and show its relation to the W-polynomial. There, we prove the first main theorem, Theorem \ref{main_thm_1}. In section 5 we introduce the Mahler measure and shown that Theorem \ref{main_thm_1} allows us to sharpen the current geometric upper bound on the Mahler measure of the Kauffman bracket of a link. In section 6 we show how to use the W-polynomial to find the form of  the Kauffman bracket for some infinite families of links. We then give a criterion under which the Mahler measure of the Kauffman brackets of links in these families diverge. In section 8 we apply the techniques from section 7 to some specific examples. Section 9 is a brief discussion of hyperbolic volume and its relation to the previous examples. We end with some directions for future work.
	
\section{The W-Polynomial}
 
 Let G be a multigraph with vertex set V(G) and edge set E(G). Each edge can be thought of as an unordered pair of vertices. The vertices of an edge refer to the elements of this unordered pair. Multi-graphs allow there to be more that one edge between two vertices and edges that connect a vertex to itself, called a loop. For all that follows we will only be concerned with graphs that come with an embedding into the plan, we will also refer to a multigraph as a graph for simplicity. 
 
 A coloring of G is a function c from the set of edges into C(G), an arbitrary set of colors. In \cite{tutte} Bollabas and Riordan defined a Tutte type polynomial for colored graphs. Here we consider a variant of this polynomial, called the W-polynomial as defined in \cite{mainpap}. 
 
  Let $G=(V,E)$ be a graph with edge $e \in E$. Denote by $G-e$ the graph whose vertex set is the same as G and whose edge set is $E-\{e\}$. Denote by $G/e$ the graph whose vertex set is found by identifying the vertices of e and whose edges set is $E-\{e\}$. These operations will be referred to as deletion and contraction respectively. The W-polynomial of G ,$W(G)(t,z_{1},z_{2})$ is a polynomial defined by the following recursive definition.
 
 1. $E_{n}$ is a graph that consists of n vertices and no edges. 
 
  $$W(E_{n})=t^{n-1}$$

 2. An edge is a bridge if deleting it increases the number of connected components of the graph. For a bridge e with $c(e)=\lambda$,
 
 $$W(G)=(x_{\lambda}+z_{1}y_{\lambda})W(G/e).$$

If e is a loop then

$$W(G)=(x_{\lambda}z_{2}+y_{\lambda})W(G-e).$$

If e is neither a bridge nor a loop

$$W(G)=x_{\lambda}W(G/e)+y_{\lambda}W(G-e).$$ 

Also, if the graph G is a disjoint union if the graphs $G_{1}$ and $G_{2}$ then 

$$W(G)=tW(G_{1})W(G_{2})$$

 Let G=(V,E) be a connected  planar graph. Give G an arbitrary order on its edges. Let F be a  spanning subtree of the graph G. For any edge in F, F-\{e\} has two connected components. As such it partitions the vertices of G into two groups. Let $Cut_{G}(F,e)$ be the collection of edges of G whose vertices are in different groups. Call an edge in F internally active with respect to F if it is the first edge in $Cut_{G}(F,e)$ according to the order on the edges of G. Call it internally inactive with respect to F otherwise. Consider an edge e not in F. Adding e to the edge set of F creates a unique cycle. Let $Cyc_{G}(F,e)$ consist of the edges in this cycle. Call e externally active with respect to F if it is the first edge in $Cyc_{G}(F,e)$ according to the order of the edges of G. Call it externally inactive with respect to F otherwise. 
 
 With these definitions we have a spanning tree expansion for the W-polynomial. Let $k(G)$ be the number of connected components in G. Then 
\begin{equation}
W(G)=t^{k(G)-1} \sum_{F} \{\prod_{IA} (x_{c(e)}+d* y_{c(e)})\}\{\prod_{EA} ( d* x_{c(e)}+y_{c(e)}))\}\{\prod_{II} x_{c(e)}\}\{\prod_{EI} y_{c(e)}\}
\end{equation}

where one sums over all spanning subtrees of G. 

Bollobas and Riordan in \cite{tutte} show that the W-polynomial is the most general graph polynomial with a spanning tree expansion. That is, any other graph polynomial with a spanning tree expansion can be realized as an evaluation of the W-polynomial.

However, there is yet  another formulation of the W-polynomial. Define the cyclomatic number of a graph G as $n(G)=|E(G)|-|V(g)|+k(G)$. Let $S \subset E(G)$, be a subset of the edges of G. Consider S as representing the spanning subgraph of G whose edge set is S. The following formulation can be found in \cite{tutte}.
 
 $$W(G)(t,z_{1},z_{2})=t^{k(G)-1}\sum_{S \subset E(G)}\{\prod_{e \in S}x_{c(e)}\} \{\prod_{e \notin S}y_{c(e)}\}z_{1}^{k(S)-k(G)}z_{2}^{n(S)}
 $$
 
We will make use of both the spanning tree expansion of the W-polynomial and the formulation just given.

\section{Graphs Associated To Knot Diagrams}

Let D be a link diagram. To this diagram we associate a signed graph. Color the plane in a checker board manner leaving the unbounded region unshaded. The vertex set of our graph will correspond to the shaded regions of the plan and the edge set will correspond to a crossing of the diagram as in Figure (\ref{fig: sgngrh}). Having constructed our graph, color the edges with a + or - according to the rule illustrated in Figure (\ref{fig: sgngrh}).

\begin{figure}[h]
\begin{center}
\includegraphics{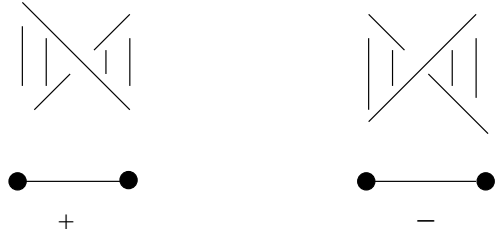}
\end{center}
\caption{Positive and Negative Edge} \label{fig: sgngrh}
\end{figure}

 Let G be a signed graph.  A \textbf{chain} in a graph is a path of edges such that each vertex interior to the path has degree 2 and the extreme vertices have degree greater than 2 or equal to 1. 
 Similarly a \textbf{sheaf} is  a collection of edges connecting the same two vertices. Say that a length of a chain or sheaf is the sum of the signs of its edges. 

\begin{figure}[h]
\begin{center}
\includegraphics{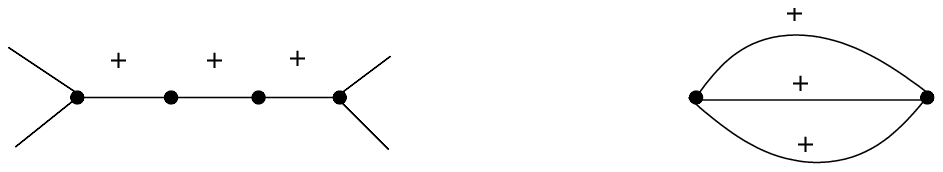}
\end{center}
\caption{A chain and sheaf of length 3} \label{fig: chain_sheaf}
\end{figure}

Let K be a link with diagram D. Two crossings in D are called twist equivalent if there is a circle in the plane of projection that crosses D at only the two crossings. The number of such equivalence classes is called the twist number of D. Furthermore we can find a diagram of K such that any two crossings that are twist equivalent are connected by a chain of bi-gons by performing a sequence of flypes if necessary.

Let D be that diagram and require that the diagram is reduced with respect to Reidemeister moves I and II.  Say that G is the signed graph associated to this diagram D. Notice that each edge in G is part of a chain, a sheaf, or represents an edge which corresponds to a crossing that is the only element in its equivalence class; such an edge must have vertices of degree greater than 2, and it may be considered a chain or a sheaf of length $\pm1$.  Notice also that there are no pendent edges and that all edges in any chain or sheaf have the same sign, lest the diagram not be reduced with respect to Reidemeister moves I and II. 

To this signed graph G we wish to associate a colored graph $\hat{G}$ that can serve as a template for all links with diagrams found by increasing the number of crossings in any twist equivalence class. Wherever there is a chain in G delete all edges and interior vertices and replace it with a single edge colored with $c$ . Similarly for each sheaf delete all edges and replace them with a single edge colored with $s$ .  By construction $\hat{G}$ has no pendent edges that are labeled with $c$ (though there may be pendent edges labeled with s), and the number of edges in $\hat{G}$ corresponds to the twist number of D. 

It will also be helpful to consider the process in reverse. Let G be a graph that has edges colored by s or c. For $t  \in \mathbb{Z}^{k}$ where k is the number of edges in G let  $\hat{G}_{t}$ be found by re-coloring the $i^{th}$ edge of G with $s_{t_{i}}$ if the original color was s and by $c_{t_{i}}$ if it was originally colored with c. To   $\hat{G}_{t}$ we can associate a signed graph by replacing an edge colored with $c_{t_{i}}$ by a chain of length $t_{i}$ and edge colored with $s_{t_{i}}$ with a sheaf of length $t_{i}$. Then as described above this signed graph corresponds to a link diagram. Call this link diagram $D_{t}$.

\begin{thm}[\cite{mainpap}] 
Let $G$ be a colored graph with color set $\{s,c\}$ and edges $\{e_{1},...e_{|E|}\}$. For any $t\in \mathbb{Z}^{|E|}$ let $D_{t}$ be the link diagram that corresponds to graph $\hat{G}_{t}$ as described above. 
\newline

If the color of edge $e_{i}$ in $\hat{G}_{t}$ is $c_{t_{i}}$ let
 
 $$x_{c_{t_{i}}}=A^{t_{i}}$$
 $$y_{c_{t_{i}}}=\frac{(-A^{-3})^{t_{i}}-A^{t_{i}}}{-A^{-2}-A^{2}}$$
 \newline
 If the color of edge $e_{i}$ in $\hat{G}_{t}$  is $s_{t_{i}}$ let
 $$x_{s_{t_{i}}}=\frac{(-A^{3})^{t_{i}}-A^{-t_{i}}}{-A^{-2}-A^{2}}$$
  $$y_{s_{t_{i}}}=A^{-t_{i}}$$
  
  $<D_{t}>=W(\hat{G}_{t})(d,d,d)$, where $d=-A^{-2}-A^{2}$.
  \label{w2kb}
  \end{thm}

 \section{ The W-polynomial and the Twist polynomial} \label{wPolyTwistpoly}
 
 In \cite{cyc} Champanerkar and Kofman defined the twist polynomial associated to a wiring diagram. In what follows we review the definition of the twist polynomial as they set forth and show that it is an evaluation of the W-polynomial.
 
A wiring diagram is a 4-valent graph with two types of vertices; horizontal and vertical. Each vertex is referred to as a twist site. Let $\tilde{L}$ be a wiring diagram with k twist sites. To each vector $w \in \mathbb{Z}^{k}$ we can construct from $\tilde{L}$ a link diagram $L_{w}$ in the following manner. Replace the $i^{th}$ twist site with $|w_{i}|$ crossings where the type of crossing and the orientation of the twist class depends on the type of vertex 
 replaced. Figures (\ref{fig: wire} a) and (\ref{fig: wire} b) below indicate how to replace a vertical, respectively horizontal, twist site with a twist equivalence class of crossings when $w_{i}=n>0$. If $w_{i}=n<0$ change all under crossings to over crossings and vice versa. 
\begin{figure}[h]
\begin{center}
\includegraphics{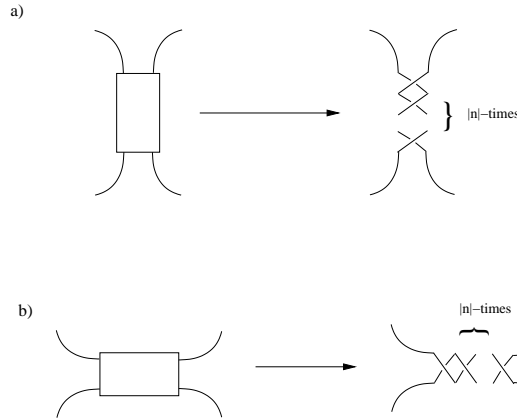}
\end{center}
\caption{(a) Filling a vertical vertex. (b) Filling a horizontal vertex} \label{fig: wire}
\end{figure}

Consider now $\sigma \in \{0,1\}^{k}$. To each $\sigma$ we associate a collection of closed components $L_{\sigma}$, called a state of the wiring diagram, constructed from $\tilde{L}$  by filling each vertex as indicated in Figure \ref{fig: clsdcmp}. In Figure (\ref{fig: clsdcmp} a) we see the corresponding fillings for the vertical vertex $v_{i}$ when $\sigma_{i}=0$  and when $\sigma_{i}=1$ respectively. In Figure (\ref{fig: clsdcmp} b) we see the same thing when vertex $v_{i}$ is horizontal. 

\begin{figure}[h]
\begin{center}
\includegraphics{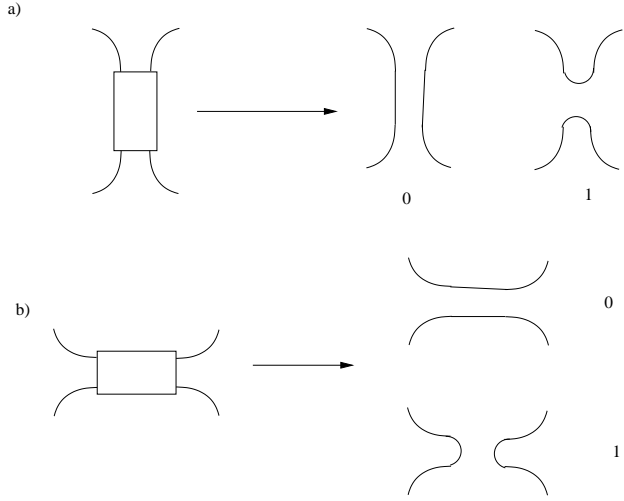}
\end{center}
\caption{Finding $L_{\sigma}$ for (a) a vertical vertex and (b) a horizontal vertex} \label{fig: clsdcmp}
\end{figure}

\begin{Def}[\cite{cyc}] 
  Let $d=-A^{-2}-A^{2}$ . For $\tilde{L}$ a wiring diagram with k open twist sites, the \textit{twist-bracket} of $\tilde{L}$ is defined by
  $$ P_{\tilde{L}}(A,x_{1},..,x_{k})=\sum_{\sigma \in \{0,1\}^{k}}(\prod_{i=1}^{k} (x_{i}-1)^{\sigma_{i}}d^{1-\sigma_{i}})<L_{\sigma}>$$
\end{Def}

\begin{thm} (\cite{cyc})
Let  $\tilde{L}$  be a wiring diagram with k open twist sites, $n \in \mathbb{Z}^{k}$, and let $\sigma(n)=\sum_{i=1..k} n_{i}$ then

$$d^{k}<L_{n}>= A^{\sigma(n)} P(A,(-A^{-4})^{n_{1}},...,(-A^{-4})^{n_{k}})$$
\end{thm}
  
   Let W be a wiring diagram. Just as we can shade a link diagram we can shade a wiring diagram. Give the wiring diagram a checkerboard shading so that the unbounded region is unshaded. Notice that the zero filling and the one filling of a vertex can be found  by the following rule. The zero filling is found by inserting any number of positive crossings according to the vertex type and giving each one a positive smoothing. The one filling is found  by inserting any number of positive crossings, giving any one of them a negative smoothing and then removing all nugatory crossings. We will call a vertex vertical or horizontal depending on how the zero filling interacts with the shading as illustrated below in Figure (\ref{fig:  verthorz}).
\begin{figure}[h]
\begin{center}
\includegraphics{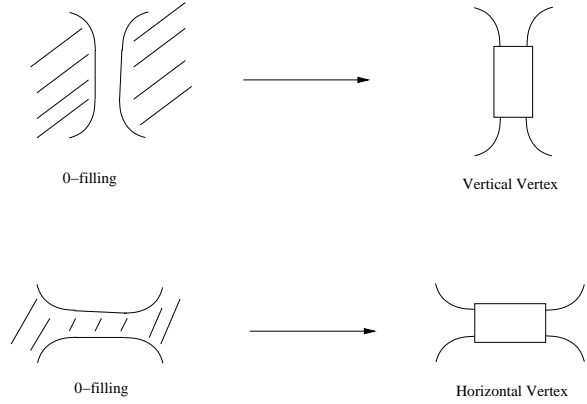}
\end{center}
\caption{Interaction between the shading and $L_{\sigma}$} \label{fig: verthorz}
\end{figure}
  
  With this definition there is an exact correspondence between wiring diagrams and colored graphs with color set $\{c,s\}$. Given a wiring diagram with a checkerboard shading we can construct a colored graph by letting vertices correspond to shaded regions and edges  correspond to vertices in the wiring diagram in the obvious way. Color an edge with an $s$ if its corresponding vertex is vertical and color the edge with a $c$ if its corresponding vertex is horizontal. Let G be a colored graph with color set $\{s,c\}$, E edges and V vertices. From G we construct a wiring diagram. First consider the link diagram $D_{\bar{1}}$, where 
$\bar{1}=(1,1,...,1) \in \mathbb{Z}^{|E|}$; that is we let all sheafs and chains have length 1. From $D_{\bar{1}}$  delete a neighborhood of each crossing, replacing that neighborhood with a vertex. Let the vertex be vertical if its corresponding edge is a sheaf and horizontal if its corresponding edge is a chain. 

We wish to see the twist polynomial of a wiring diagram as the W-polynomial of its corresponding colored graph. We must specify how subgraphs of a graph correspond to particular states of a wiring diagram. 
Let $G_{W}$ and $W_{G}$ be a graph and corresponding wiring diagram. Say that $G_{W}$ has E edges with an arbitrary order,$\{e_{1},\dots,e_{E}\}$. This of course gives a corresponding order on the vertices of the wiring diagram $W_{G}$. 

Consider the following correspondence between subgraphs of G and vectors in $\{0,1\}^{E}$. Let $S \subset E(G)$, and consider it as representing the spanning subgraph of G with edge set S as intended when considering the W-polynomial. If the $i^{th}$ edge e is a sheaf and $e\in S$ let $\sigma_{i}=1$ while if it is not in S, let $\sigma_{i}=0$. Similarly if e is a chain and $e \in S$ let $\sigma_{i}=0$  and if e is not in S let  $\sigma_{i}=1$.  This exhibits a particular one-to-one correspondence between spanning subgraphs of G and states of the wiring diagram.

\begin{rmk} \label{sgnchc}\end{rmk}There is a difference in the sign conventions for sheafs in a graph and the corresponding vertical vertices in a wiring diagram. For a graph $G$ with $|E|$ edges if 
$t  \in \mathbb{Z}^{|E|}$, and $W_{G}$ is its corresponding wiring diagram then if 
$r  \in \mathbb{Z}^{|E|}$ with $r_{i}=t_{i}$ if the $i^{th}$ edge/vertex is a chain/horizontal and $r_{i}=-t_{i}$  if the $i^{th}$ edge/vertex is a sheaf/vertical, the $D_{t}$ and $(W_{G})_{r}$ are the same diagram.

\begin{figure}[h]
\begin{center}
\includegraphics{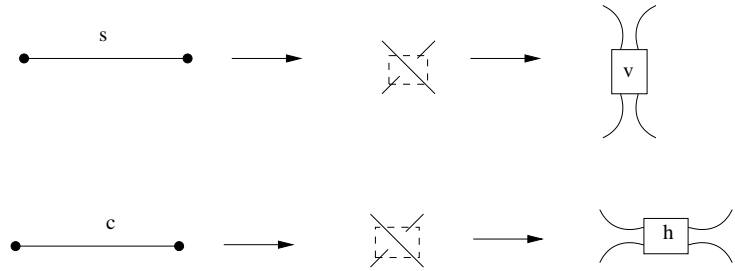}
\end{center}
\caption{Colored graph becomes a wiring diagram} \label{fig: g2w}
\end{figure}

\begin{lem}
Let $G_{W}$ and $W_{G}$ be a graph and corresponding wiring diagram.  Say that $S \subset E(G)$ represents the spanning subgraph of $G_{W}$ with edge set S. Let $\sigma \in \{0,1\}^{E}$  be as described above. Then 

$$d^{|S|+2k(S)-V-1}=<(W_{G})_{\sigma}>$$.
\end{lem}

\begin{proof}
 By construction a sheaf in S  corresponds to a 1 filling of the appropriate vertical vertex while a chain in S correspond to a 0 filling of the appropriate horizontal vertex. As indicated in Figure \ref{fig: rgnb} closed components of $(W_{G})_{\sigma}$ are exactly the boundary of a regular neighborhood of connected components of S. 
 
 \begin{figure}[h]
\begin{center}
\includegraphics{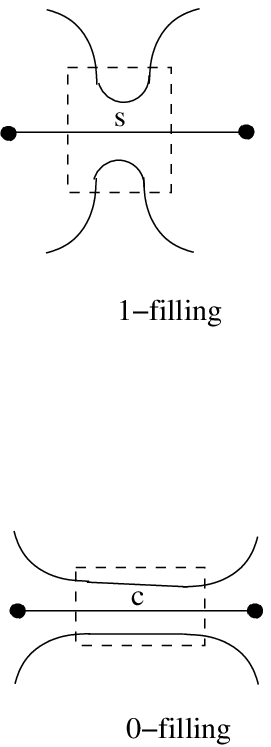}
\end{center}
\caption{Top: 1 filling of a vertical vertex. Bottom: 0 filling of a horizontal vertex} \label{fig: rgnb}
\end{figure}
 Say S has connected components $c_{1},...,c_{k(S)}$. Then let $V_{i}$, $S_{i}$, and $F_{i}$ be the number of vertices, edges and faces of $c_{i}$ respectively. Since G is planar so is each $c_{i}$ and thus $ F_{i}=S_{i}-V_{i}+2$. Then $\sum_{i=1}^{k(S)} F_{i}= |S|-V+2k(S)$.
 
  Notice that components of $(W_{G})_{s}$ are exactly the faces of the connected components. Thus the lemma follows.
 \end{proof}

\begin{thm}
Let $G_{W}$ and $W_{G}$ be a colored graph and corresponding wiring diagram. Say that they are connected. Give the edges of $G_{W}$ an arbitrary order, $\{e_{1},\dots,e_{E}\}$. Say that $G_{W}$ has V vertices. Augment the coloring of $G_{W}$ so that the $i^{th}$ edge $G_{W}$ is colored with either $c_{i}$ or $s_{i}$ depending on whether the edge was colored with  $c$ or $s$. Call this new graph $\tilde{G}_{W}$.  Then with the evaluation
\begin{eqnarray}
x_{c_{i}}&=&1\\
y_ {c_{i}}&=&\frac{x_{i}-1}{d}\\
x_{s_{i}}&=&\frac{x_{i}-1}{d}\\
y_{s_{i}}&=&1
\end{eqnarray}
$$ d^{E} W(\tilde{G}_{W})(d,d,d)=P(A,x_{1},...x_{E})$$
\end{thm}
\begin{proof}
\begin{eqnarray} \nonumber
&&d^{E} W( G )(d,d,d)=\sum_{S \subset E(G)} \{ \prod_{e \in S} x_{c(e)}\} \{\prod_{e \notin S} y_{c(e)}\}d^{E} d^{V-|S|+2 k(S)-1}\\ \nonumber
&= &\sum_{S \subset E(G)}\{ \prod_{e_{i} \in S\ c(e_{i})=c_{i} } x_{c(e)}\} \{ \prod_{e_{i} \in S\ c(e_{i})=s_{i} } x_{c(e)}\}\{\prod_{e_{i} \notin S c(e_{i})=c_{i}  } y_{c(e)}\}\{\prod_{e_{i} \notin S  c(e_{i})=s_{i} } y_{c(e)}\} d^{|E|} d^{|V|-|S|+2k(S)-1} \\ \nonumber
&=&\sum_{S \subset E(G)} \{\prod_{ e_{i} \in S \ c(e_{i})=s_{i}} \frac{x_{i}-1}{d}\}\{\prod_{ e_{i} \notin S \ c(e_{i})=c_{i}} \frac{x_{i}-1}{d}\} d^{E} d^{V-|S|+2k(S)-1}\\ \nonumber
&=& \sum_{\sigma \in \{0,1\}^{E}} \{\prod_{i=1}^{E} (x_{i}-1)^{\sigma_{i}}\}d^{E -\sum_{i=1}^{E} \sigma_{i}} <(W_{G})_{\sigma}> \\ \nonumber
&=&  \sum_{\sigma \in \{0,1\}^{E}} \{\prod_{i=1}^{E} (x_{i}-1)^{\sigma_{i}} d^{1-\sigma_{i}}\} <(W_{G})_{\sigma}> \\ \nonumber
&=&P_{W_{g}}(A,x_{1}, \dots, x_{E})
\end{eqnarray}
The fourth equality follows from the previous lemma and the fact that $\sigma_{i}=1$ for sheafs in S and chains not in S and $\sigma_{i}=0$ otherwise.
\end{proof}

Theorem 2 in \cite{cyc} was proven using an expression for a twist in $TL_{2}$ and the fact that it is the center of this ring. Theorem 1 from \cite{mainpap} and Theorem 3 allow an alternate proof. Here we prove the theorem in this context.

\begin{proof} (Theorem 2)
 Let $G_{W}$ be the colored graph associated to the wiring diagram $\tilde{L}$. Note that $G_{W}$ has k edges. Let $r\in \mathbb{Z}^{k}$ and let $t\in\mathbb{Z}^{k}$ such that $D_{t}$ as constructed from $G_{W}$ is the same diagram $\tilde{L}_{r}$. Consider the evaluation of W-polynomial as indicated in Theorem 1. Then
\begin{eqnarray} \nonumber
& &d^{k} <D_{t}>=d^{k} W( \hat{G}_{t} )(d,d,d)\\ \nonumber
&=&d^{k-|V(G)|-1}\sum_{S \subset E(G)} \{ \prod_{e \in S} x_{c(e)}\} \{\prod_{e \notin S} y_{c(e)}\} d^{(|S|+2 k(S))} \\ \nonumber
&= &\sum_{S \subset E(G)}\{ \prod_{chains \in S } x_{c(e)}\} \{ \prod_{sheafs \in S } x_{c(e)}\}\\ \nonumber
& & \{\prod_{chains \notin S } y_{c(e)}\}\{\prod_{sheafs \notin S } y_{c(e)}\} d^{k} d^{|V|-|S|+2k(S)-1} \\ \nonumber
 &= & \sum_{S \subset E(G)} \{ \prod_{e_{i} \in S \ chain }A^{t_{i}}\} \{ \prod_{e_{i} \in S \ sheaf } 
\frac{(-A^{3})^{t_{i}}-A^{-t_{i}}}{d}\}\\ \nonumber
 & & \{\prod_{e_{i} \notin S \ chains } \frac{(-A^{-3})^{t_{i}}-A^{t_{i}}}{d}\}\{\prod_{e_{i} \notin S \ sheaf} A^{-t_{i}}\} d^{k} d^{|V|-|S|+2k(S)-1} \\ \nonumber
 &=&  \{\prod_{e_{i} \in G \ chain} A^{t_{i}}\}\{\prod_{e_{i} \in G \ sheaf} A^{-t_{i}}\} \\ \nonumber & &
\sum_{S \subset E(G)} \{ \prod_{e_{i} \in S \ sheaf } 
\frac{(-A^{4})^{t_{i}}-1}{d}\} \{\prod_{e_{i} \notin S \ chains } \frac{(-A^{-4})^{t_{i}}-1}{d}\} d^{k} d^{|V|-|S|+2k(S)-1} \\ \nonumber
&=&A^{\sum_{i=1}^{k} r_{i}}\sum_{\sigma \in \{0,1\}^{E}} \prod_{i=1}^{E}( \frac{(-A^{-4})^{r_{i}}-1}{d})^{\sigma_{i}} d^{k} <(W_{G})_{\sigma}> \\ \nonumber
&=& A^{\sigma(r)} \sum_{\sigma \in \{0,1\}^{k}}  \prod_{i=1}^{E} ((-A^{-4})^{r_{i}}-1)^{\sigma_{i}} d^{k-\sum_{i=1}^{k} \sigma_{i}}   <(W_{G})_{\sigma}>  \\ \nonumber
&=& A^{\sigma(r)} \sum_{\sigma \in \{0,1\}^{k}}  \prod_{i=1}^{E} ((-A^{-4})^{r_{i}}-1)^{\sigma_{i}} d^{\sum_{i=1}^{k}1- \sigma_{i}}  <(W_{G})_{\sigma}>  \\ \nonumber
&=&A^{\sigma(r)}  \sum_{\sigma \in \{0,1\}^{k}} ( \prod_{i=1}^{E}((-A^{-4})^{r_{i}}-1)^{\sigma_{i}} d^{1-\sigma_{i}})  <(W_{G})_{\sigma}> \\ \nonumber
&=&A^{\sigma(r)} P_{\tilde{L}}(A,(-A^{-4})^{r_{i}},\dots,(-A^{-4})^{r_{E}})
 \end{eqnarray}
 
 The sixth equality follows from Lemma 1 and the following two facts: 1)  $\sigma_{i}=1$ for sheafs in S and chains not in S and $\sigma_{i}=0$ otherwise, and 2)  the sign convention is opposite when comparing the crossings in a sheaf to the crossing in its  corresponding filling and the same when comparing the crossings in a chain and the crossings in its corresponding filling. (As noted in Remark \ref{sgnchc}).
\end{proof}

\begin{cor}
Let K be a link with twist number n and diagram D realizing that twist number.
Then $d^{n} <D_{t}>$ is a polynomial in n+1 variables for all $t \in \mathbb{Z}^{k}$. 
\end{cor}

The spanning tree expansion of the W-polynomial allows a slight sharpening of this result.

\begin{thm} \label{main_thm_1}
 Let L be a link with diagram D such that twist number of D realizes the twist number of K. Let G be the colored graph with E edges V vertices and an order on the edges so that $\hat{G}_{t}=D$  for $t \in \mathbb{Z}^{k}$.  For any spanning subtree F let $p_{F}=\#\{sheafs \in F\}+\#\{chains \notin F\}$ and $p=max_{F}(p_{F})$. Then
 $d^{p}<D_{t}>$ is a polynomial in k+1 variables for all $t \in \mathbb{Z}^{k}$.  Furthermore $p\leq$ the twist number of L.
 \end{thm}

\begin{proof}

Consider $W(G_{t})$ evaluated as in Theorem 1. Each edge in $\hat{G}_{t}$ is a chain or sheaf. Each spanning subtree classifies the edges according to internal/external and active/inactive. As such there are eight different types of edges as one sums over spanning subtrees. Accordingly for each spanning subtree we have a partition of $\{1,2, \ldots, k\}$ corresponding to these 8 different groups, where k is the twist number of D. Let each set of the partition be denoted by the corresponding word of the appropriate form; (I/E)(A/I)(C/S).
   
 \begin{eqnarray} W(\hat{G}_{t})&=& \nonumber
 \sum_{F} \{\prod_{i \in IAC} (x_{c_{t_{i}}}+d y_{c_{t_{i}}})\}
 \{\prod_{i \in IAS} (x_{s_{t_{i}}}+d y_{s_{t_{i}}})\}  
 \{\prod_{i \in EAC} ( d x_{c_{t_{i}}}+y_{c_{t_{i}}})\} \\ \nonumber
  & &\{\prod_{i \in EAS} ( d x_{s_{t_{i}}}+y_{s_{t_{i}}})\}
 \{\prod_{i \in IIC} x_{c_{t_{i}}}\} 
 \{\prod_{i \in IIS} x_{s_{t_{i}}}\}
 \{\prod_{i \in EIC} y_{c_{t_{i}}}\}
 \{\prod_{i \in EIS} y_{s_{t_{i}}}\} \nonumber
 \end{eqnarray}
 
 Apply the evaluation given in Theorem 1. Notice that $x_{c_{t_{i}}}+d y_{c_{t_{i}}}=(-A^{-3})^{t_{i}}$ and $d x_{s_{t_{i}}}+y_{s_{t_{i}}}=(-A^{3})^{t_{i}}$.
\begin{eqnarray} \nonumber
<D_{t}> &=& \sum_{F} \{\prod_{i \in IAC} (-A^{3})^{-t_{i}}\}
  \{\prod_{i \in IIC} A^{t_{i}}\}
   \{\prod_{i \in EAS} (-A^{3})^{t_{i}}\}
   \{\prod_{i \in EIS}A^{-t_{i}}\}\\ \nonumber
 & &\{\prod_{i \in IAS} (\frac{(-A^{3})^{t_{i}}-A^{-t_{i}}}{d}+d A^{-t_{i}})\}
 \{\prod_{i \in EAC} ( d A^{t_{i}}+\frac{(-A^{-3})^{t_{i}}-A^{t_{i}}}{d}))\} \\ \nonumber
 & &\{\prod_{i \in IIS}\frac{(-A^{3})^{t_{i}}-A^{-t_{i}}}{d}\}
 \{\prod_{i \in EIC}\frac{ (-A^{3})^{t_{i}}-A^{t_{i}}}{d}\} \nonumber
 \end{eqnarray}
 \pagebreak
 \begin{eqnarray} \nonumber
 &=&\sum_{F} d^{-(p_{F})}  (-A^{3})^{(\sum_{EAS} t_{i}-\sum_{IAC}t_{i})} A^{(\sum_{IIC}t_{i}-\sum_{EIS} t_{i})} \\ \nonumber
& &  \{\prod_{i \in IAS} (-A^{3})^{t_{i}}+A^{-t_{i}}(d^{2}-1)\}\{\prod_{i \in IIS}  (-A^{3})^{t_{i}}-A^{-t_{i}}\}\\ \nonumber
 & & \{\prod_{i \in EAC} (-A^{-3})^{t_{i}}+A^{t_{i}}(d^{2}-1)\}\{\prod_{i \in EIC}(-A^{-3})^{t_{i}}-A^{t_{i}}\} \\ \nonumber
 &=&\sum_{F} d^{-(p_{F})}  (-A^{3})^{(\sum_{EAS} t_{i}-\sum_{IAC}t_{i})} A^{(\sum_{IIC}t_{i}-\sum_{EIS} t_{i})} \\ \nonumber
 & & A^{-\sum_{IS} t_{i}}  \{\prod_{i \in IAS} (-A^{4})^{t_{i}}+(d^{2}-1)\}\{\prod_{i \in IIS}  (-A^{4})^{t_{i}}-1\}\\ \nonumber
 & & A^{\sum_{EC}t_{i}}  \{\prod_{i \in EAC} (-A^{-4})^{t_{i}}+(d^{2}-1)\}\{\prod_{i \in EIC}(-A^{-4})^{t_{i}}-1\} \\ \nonumber 
 &=&\sum_{F} d^{-(p_{F})}  (-A^{3})^{\sum_{EAS} t_{i}-\sum_{IAC}t_{i}} A^{\sum_{IIC}t_{i}-\sum_{EIS} t_{i}+\sum_{EC}t_{i}-\sum_{IS} t_{i}} \\ \nonumber
  & & \{\prod_{i \in IAS} (-A^{4})^{t_{i}}+(d^{2}-1)\}\{\prod_{i \in IIS}
  (-A^{4})^{t_{i}}-1\}\\ \nonumber
 & & \{\prod_{i \in EAC} (-A^{-4})^{t_{i}}+(d^{2}-1)\}\{\prod_{i \in EIC}(-A^{-4})^{t_{i}}-1\}\\ \nonumber
  \end{eqnarray}
   
 Then $d^{p} <D_{t}>$ is a polynomial in k+1 variables for all $t \in \mathbb{Z}^{k}$. Furthermore if there is a spanning subtree consisting of all the sheafs in $\hat{G}_{t}$ then $p=$twist number of K. Otherwise, $p <$ twist number of K. 
 \end{proof}
\section{Mahler Measure}

\begin{Def} Let $f\in \mathbb{C}[z_{1}^{\pm1},\dots, z_{t}^{\pm1}]$. The Mahler measure of f is
	
	$$M(f)=exp \int_{0}^{1} \ldots \int_{0}^{1} log|f(e^{2 \pi i \theta_{1}}, \ldots , e^{2 \pi i \theta_{t}})| d\theta_{1} \ldots d\theta_{t}$$
	\end{Def}

For a Laurent polynomial with a single variable, $f(z):=a_{0}z^{k}\prod_{i=1}^{n}(a_{i}-z)$ Jensen's formula shows that 

$$M(f)=|a_{0}|\prod_{i=1}^{n}max(1,|a_{i}|)$$

One may also consider the Euclidean Mahler measure $M_{e}(f):=\prod_{i=1}^{n}max(1,|a_{i}|)$. Note also that the (Euclidean) Mahler measure is multiplicative. 

\begin{rmk} \label{norm} \end{rmk}  Let $||f||$ denote the $L^{2}$ norm of $f$. Then Schinzel shows in \cite{mm} that $M(f) < ||f||$. This observation allows us to control the Mahler measure of  the Kauffman bracket by controlling its coefficients.

\begin{cor}
 For all $t \in \mathbb{Z}^{k}$, $||d^{p} <D_{t}>|| 
< C$ for  some positive constant C. Furthermore since $d$ is cyclotomic,$$M(<D_{t}>)=M(d^{p} <D_{t}>)<||d^{p} <D_{t}>|| < C$$ 
 \end{cor}
 
\begin{proof}
The corollary follows directly from the fact that $d^{p} <D_{t}>$ is a polynomial in $k+1$ variables and Remark \ref{norm}.
\end{proof}

\begin{rmk}\end{rmk}
Any diagram of a link that realizes the twist number of the link and consists of only chains is an example where p is strictly less than the twist number of the link. Pretzel links provide a specific example.

\section{Calculations with the W-polynomial}

In \cite{mainpap} Jin and Zhang use the W-polynomial to find a method for calculating the Kauffman bracket of rational links. Here we apply the same technique for a different class of links.

Suppose that we repeatedly perform a rational surgery R on n unknotted components of a diagram D where the component on which we do our surgery intersects the plane of projection in two shaded regions as in Figure \ref{fig: rptsrg}. This in turn corresponds to repeatedly adding a graph to the two vertices representing the indicated shaded regions as in Figure \ref{fig: grphadd}. Notice that we can twist the two strands in the diagram by repeatedly performing a $\pm1/2$ rational surgery. 

\begin{figure}[h]
\begin{center}
\includegraphics{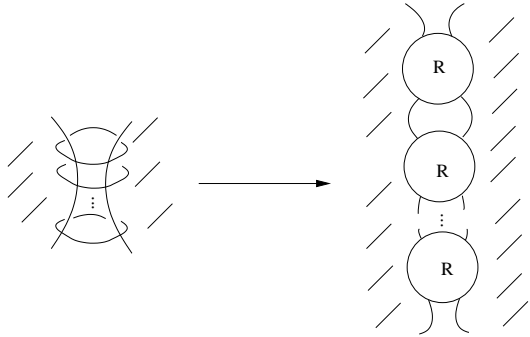}
\end{center}
\caption{Repeated rational surgery} \label{fig: rptsrg}
\end{figure}

\begin{figure}[h]
\begin{center}
\includegraphics{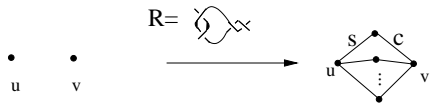}
\end{center}
\caption{Repeated graph additions} \label{fig:  grphadd}
\end{figure}

Let $G$ be the graph that corresponds to the original diagram $D$ and name the two pictured vertices $u$ and $v$. Let $G_{i}$ be the graph that corresponds to the diagram found by adding $i$ copies of the graph of the rational tangle. Here we want to consider the following form of the W-polynomial.
 
$$W(G)(d,d,d)=d^{-k(G)-|V(G)|} \sum_{S \subset E(G)}\{ \prod_{e\in S} x_{f(e)}\}\{ \prod_{s \notin S} y_{f(e)}\} d^{|S|+2 k<S>}$$

Let $S \subset E(G)$ represent the spanning subgraph of $G$ with edge set S as previously defined. Say that S is in state 1 if $u$ and $v$ are in different connected components and say that it is in state 2 if they are in the same component, represented as $S:1$ or $S:2$ respectively. For $G_{i}$,  $E(G_{i})=E(G_{i-1}) \cup E(C)$  and $V(G_{i})=v(G_{i}) \cup V(C)-\{u,v\}$, where C is the graph of the rational tangle R and we understand $u$ and $v$ to represent the vertices in C that are identified with the vertices $u$ and $v$ in $G_{i-1}$. Then

\begin{eqnarray} \nonumber
& &W(G_{i})(d,d,d)=d^{-k(G_{i})-|V(G_{i})|} \sum_{S \subset G_{i}}\{ \prod_{e\in S} x_{f(e)}\}\{ \prod_{s \notin S} y_{f(e)}\} d^{|S|+2 k<S>}\\ \nonumber
&=&d^{-k(G_{i})-|V(G_{i})|} \sum_{S \subset G_{i-1} \cup C }\{ \prod_{e\in S} x_{f(e)}\}\{ \prod_{s \notin S} y_{f(e)}\} d^{|S|+2 k<S>}\\ \nonumber
&=&d^{-k(G_{i})-|V(G_{i})|} \sum_{T \subset G_{i-1} }\{ \prod_{e\in T} x_{f(e)}\}\{ \prod_{e \notin T} y_{f(e)}\} d^{|T|+2 k<T>}  \sum_{W \subset C}\{ \prod_{e\in W} x_{f(e)}\}\{ \prod_{e \notin W} y_{f(e)}\} d^{|S|+2 \delta(T,W)}\\ \nonumber
&=&d^{-k(G_{i})-|V(G_{i})|} (\sum_{T:1 \subset G_{i-1} }\{ \prod_{e\in T} x_{f(e)}\}\{ \prod_{s \notin T} y_{f(e)}\} d^{|T|+2 k<T>}+\sum_{T:2 \subset G_{i-1} }\{ \prod_{e\in T} x_{f(e)}\}\{ \prod_{s \notin T} y_{f(e)}\} d^{|T|+2 k<T>}) \\ \nonumber & & \sum_{W \subset C}\{ \prod_{e\in W} x_{f(e)}\}\{ \prod_{e \notin W} y_{f(e)}\} d^{|W|+2 \delta(T,W)} \nonumber
\end{eqnarray}
Where $\delta(T,W)=k<T\cup W>-k<T>$ is the change in the number of connected components after adding the subgraph of C to T, the subgraph of G. For subgraphs $T$ and $W$ as above, if we know the state of $T$ and the state of $T\cup W$ then we can find $\delta(T,W)$. In fact

\begin{lem}
If

 $T:1$ and $(T \cup W):1$ then $\delta(T,W)=k<W> -2$. 
 
 $T:1$ and $(T \cup W):2$ then $\delta(T,W)=k<W> -2$.
 
 $T:2$ and $(T \cup W):1$ then $\delta(T,W)=k<W> -2$.
 
 $T:2$ and $(T\cup W):2$ then $\delta(T,W)=k<W>-1$
 
 \label{delta}
 \end{lem}
 
\begin{proof}
Simply count the number of components in W that are not already counted by components of T.
\end{proof}

Partition the subgraphs S of C into 4 types depending on how it changes the state of the spanning subgraph T of $G_{i-1}$. T can be in  state 1 or state 2 and $T \cup S$, $S \subset C$ can be in state 1 or state 2. Categorize subgraphs S of C accordingly. Then let
\begin{eqnarray} \nonumber
a_{1,1}&=&\sum_{S \subset C\ s.t\ T:1 \& T\cup S:1} \{ \prod_{e\in S} x_{f(e)}\}\{ \prod_{s \notin S} y_{f(e)}\} d^{|S|+2 \delta(T,S)}\\ \nonumber
a_{1,2}&=&\sum_{S \subset C\ s.t\ T:1 \& T\cup S:2}  \{ \prod_{e\in S} x_{f(e)}\}\{ \prod_{s \notin S} y_{f(e)}\} d^{|S|+2 \delta(T,S)}\\ \nonumber
a_{2,1}&=&\sum_{S \subset C\ s.t\ T:2 \& T\cup S:1} \{ \prod_{e\in S} x_{f(e)}\}\{ \prod_{s \notin S} y_{f(e)}\} d^{|S|+2 \delta(T,S)}\\ \nonumber
a_{2,2}&=&\sum_{S \subset C\ s.t\ T:2 \& T\cup S:2}  \{ \prod_{e\in S} x_{f(e)}\}\{ \prod_{s \notin S} y_{f(e)}\} d^{|S|+2 \delta(T,S)}\\ \nonumber
S_{1}&=&\sum_{S:1 \subset G_{i-1} }\{ \prod_{e\in S} x_{f(e)}\}\{ \prod_{s \notin S} y_{f(e)}\} d^{|S|+2 k<S>} \\ \nonumber
S_{2}&=&\sum_{S:2 \subset G_{i-1} }\{ \prod_{e\in S} x_{f(e)}\}\{ \prod_{s \notin S} y_{f(e)}\} d^{|S|+2 k<S>} \nonumber
\end{eqnarray}

$$W(G_{i})=d^{k<G_{i}>-|V(G_{i})|} \begin{pmatrix} S_{1} & S_{2}\end{pmatrix} \begin{pmatrix} a_{1,1}&a_{1,2}\\ a_{2,1}& a_{2,2}\end{pmatrix} \begin{pmatrix}1\\1\end{pmatrix}$$

 Proceed recursively  to 

$$W(G_{i})=d^{k<G_{i}>-|V(G_{i})|}\begin{pmatrix}S_{1}&S_{2}\end{pmatrix}  \begin{pmatrix} a_{1,1}&a_{1,2}\\ a_{2,1}& a_{2,2}\end{pmatrix}^{i} \begin{pmatrix}1\\1\end{pmatrix}$$
where $S_{1}$ and $S_{2}$ are the contributions to the W-polynomial of the graph G (before any surgeries are performed) where vertices $u$ and $v$ are in the same receptively different components.

\begin{lem}
In the above situation $a_{2,1}=0$ and $a_{2,2}=a_{1,1}+a_{1,2}*d^{2}$.
\end{lem}

\begin{proof}
After adding a subgraph of C to a subgraph of $G_{i-1}$ $u$ and $v$ are either in the same component or in different components. Thus if we sum along any row of $a$ we sum over all subgraphs of C. This observation along with Lemma \ref{delta} completes the proof.
\end{proof}

\begin{lem}
$$\begin{pmatrix}a_{1,1}&a_{1,2}\\0&a_{1,1}+d^2\ a_{1,2}\end{pmatrix}^{n} =\begin{pmatrix} a_{1,1}^{n} & d^{-2} \sum_{j=1}^{n} {n\choose j} a_{1,1}^{n-j}\ (a_{1,2} d^{2})^{j} \\0&(a_{1,1}+d^{2}\ a_{1,2})^{n}\end{pmatrix}$$
\end{lem}
\begin{proof}
Proceed by induction
\end{proof}

 \begin{thm} \label{KBform}
 After n rational surgeries as described above
 
 $$W(G_{n})=d^{k<G_{n}>-|V(G_{n})|} ( a_{1,1}^n (S_{1}(1-d^{-2}))+(a_{1,1}+d^{2} a_{1,2})^{n}(S_{1}*d^{-2}+S_{2}))$$
 \end{thm} 
\begin{proof}
The proof follows from the previous lemma after a change of basis.
\end{proof}

\begin{Ex} \end{Ex}
A twist link with $2n$ crossings can be constructed by repeatedly performing -1/2 surgery as indicated in the Figure \ref{fig: rptgrp1}. This corresponds to adding a sheaf of length two repeatedly to a pair of vertices. 

\begin{figure}[h]
\begin{center}
\includegraphics{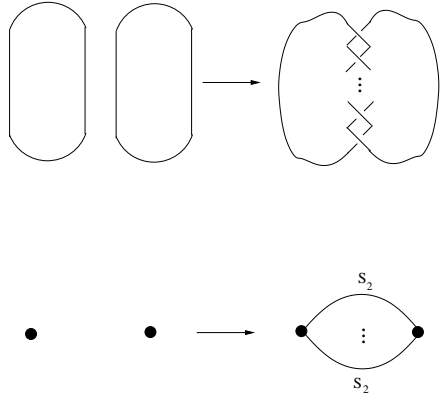}
\end{center}
\caption{Repeated graph additions} \label{fig: rptgrp1}
\end{figure}

In this case the graph G is two vertices with no edges. We see that there are no subgraphs of G in which the two vertices are connected so, $S_{1}=d$ and $S_{2}=0$. Then $a_{1,1}=y_{s_{2}}$ and $a_{1,2}=d^{-1} x_{s_{2}}$.  Let $G_{n}$ be the graph with two vertices and n edges, each a sheaf of length 2. Then

$$W(G_{n})(d,d,d)=d^{-2}(y_{s_{2}}^{n}(d^{2}-1)+(y_{s_{2}}+d x_{s_{2}})^{n}$$
 Evaluating the colors of the edges as in Theorem 1 we find the Kauffman bracket of the twist link with $2n$ crossings $T_{2n}$ is
$$<T_{2n}>=d^{-2}(A^{-2n}(A^{4}+1+A^{-4})+(-1)^{n}(A^{-3n})$$

\begin{Ex} \end{Ex}

Any Montesinos link where each rational tangle is the same can be seen as the result of some number of repeated rational surgeries on two unlinked simple closed components. As such its Kauffman bracket will have the form
$$d^{b}(a_{1,1}^{n}(A^{4}+1+A^{-4})+(a_{1,1} +d^{2} a_{1,2})^{n})$$

In particular if we consider a pretzel link then $a_{1,1}=y_{c_{n}}$ and $a_{1,2}=d^{-1} x_{c_{n}}$. The pretzel link $(n,n, \ldots,n)$ with m twist equivalence classes will have Kauffman bracket 

$$ <P_{m,n}>=d^{-2}( (\frac{(-A^{-3})^{n}-A^{n}}{d})^{m}(A^{4}+1+A^{-4})+(\frac{(-A^{-3})^{n}-A^{n}}{d} +d A^{n})^{m})$$

\begin{figure}[h]
\begin{center}
\includegraphics{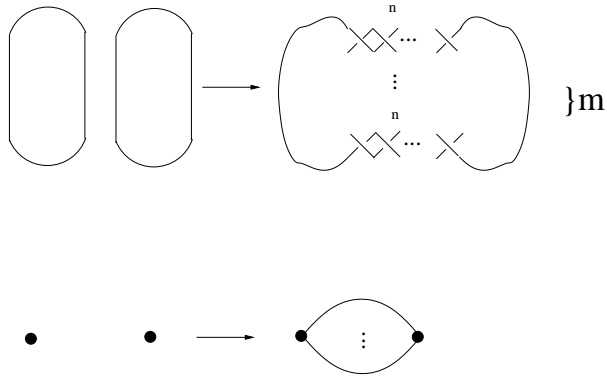}
\end{center}
\caption{Repeated graph additions} \label{fig: rptgrp}
\end{figure}

\section{Euclidean Mahler measure and repeated rational surgeries}

 In \cite{Twisting} and \cite{silver}, the (Euclidean) Mahler Measure of the Jones polynomial was investigated under twisting, while in \cite{silver_williams} the Mahler measure of the multivariable Alexander polynomial was the subject. There are also several attempts such as \cite{zeros}, \cite{chang_schrock} to find the locus of zeros for infinite families of knots and links. Also in the physics community there are many attempts to find the partition function of Potts model spin systems in an when trying to understand phase changes in the thermodynamic limit. Here some of those techniques are applied to investigate the behavior of the Mahler measure under repeated  rational surgeries. 
 
 Recall that the Mahler measure of a polynomial $f(z):=a_{0}z^{k}\prod_{i=1}^{n}(a_{i}-z)$ can be found by

$$M(f)=|a_{0}|\prod_{i=1}^{n}max(1,|a_{i}|)$$
 
 The following theorem is concerned with precisely understanding the locus of zeros for an infinite family of polynomials. It was originally proven in \cite{beraha_kahane_weiss}, and extended in \cite{sokal_1}. This statement of the theorem  comes from \cite{sokal_2} as this is its most convenient form for our current purpose.
 
 Consider a family of polynomials $P_{n}(z)=\sum_{i=1}^{k} a_{i}(z) \lambda_{i}(z)^{n}$, where each $a_{i}$ and $\lambda_{i}$ are analytic in D, some region in the complex plane. 
 Let $\mathcal{Z}(P_{n})=\{z \in \mathbb{C}\ s.t. P_{n}(z)=0\}$. We are interested in the locus of zeros for the whole family. With that in mind let
 \newline
 
  \textit{lim inf}($\mathcal{Z}(P))$= \{z $\in$ D s.t. every neighborhood U of $z$ has nonempty intersection with all but finitely many of the sets  $\mathcal{Z}(P_{n}))$\}
  \newline
  
  \textit{lim sup}($\mathcal{Z}(P))$=\{ z $\in$ D s.t. every neighborhood of $z$ has nonempty intersection with infinitely many sets $\mathcal{Z}(P_{n})$\}
 \newline
 
 Call an index k dominant at $z$ if $ |\lambda_{k}(z)| \geq |\lambda_{l}|,\ l=1\ldots k$.
 
  Let $D_{k} =\{z\ s.t. \lambda_{k} \text{ is dominant at z}\}$.
 
 \begin{thm}(\cite{beraha_kahane_weiss} ,\cite{sokal_1},\cite{sokal_2}) 
 
Let D be a domain in $\mathbb{C}$ and let $a_{1},\dots,a_{k}$ and $\lambda_{1}, \dots,\lambda_{k}$ be analytic functions in D, with  $k \geq 2$. Assume also that none of these functions are identically zero, and that there is no pair of indices p and q such that 
 $\lambda_{p}=\omega \lambda_{q}$ for some constant with $|\omega|=1$ and 
 $D_{p}(=D_{q})$ is nonempty. Then
 \newline
 
 \textit{lim inf}($\mathcal{Z}(P))$= \textit{lim sup}($\mathcal{Z}(P)$) and a point z lies in this set if and only if
 \newline
 
 a) there is a unique dominant index p at z and $a_{p}(z)=0$ or
 \newline
 
 b) there are two or more dominant indices at z
 \label{eqcrv}
\end{thm}

Property (a) gives isolated points and property  (b) gives the equi-modular curves of a family of polynomials or isolated points.  Here we consider those points satisfying property (b).

\begin{cor}
  Consider a family of polynomials, $P_{n}(z)=a_{1}(z)\  \lambda_{1}(z)^{n}+a_{2}\ \lambda_{2}(z)^{n}$, where the above conditions are satisfied. Then the zeros of this family converge to some isolated points and to the equi-modular curve $\mathcal{B}$ where $|\lambda_{1}|=|\lambda_{2}|\neq 0$.
 \end{cor}
 
 \begin{thm}\label{zeros}
  Let $P_{n}(z)=a_{1}(z)\  \lambda_{1}(z)^{n}+a_{2}\ \lambda_{2}(z)^{n}$ be a family of polynomials that satisfy the requirement of the above theorem. Say that $\mathcal{B}$ is the curve where $|\lambda_{1}(z)|=|\lambda_{2}(z)|\neq 0$. If there is any non-isolated point of $\mathcal{B}$ that lies outside the unit circle then the euclidean Mahler measure of $P_{n}$ diverges as n goes to infinity. 
 \end{thm}
 \begin{proof}
 Let $z_{1}$ be such a point on $\mathcal{B}$ and let $U_{1}$ be a neighborhood of $z_{1}$ that lies completely outside the unit circle, and let $m=min(|x|: x\in U_{1})$. For 
 $i \in \{2,3,\ldots \}$ choose $i$ points $z_{i,1},\dots, z_{i,i}$ on $\mathcal{B}$ and neighborhoods
  $U_{i,1},\ldots,U_{i,i}$ such that each neighborhood is disjoint and contained in $U_{1}$. Each of these points is a limit point of the zeros for this family of polynomials. Thus there is some first polynomial, $P_{n_{i}}(z)$ to have a zero in each of the neighborhoods 
 $U_{i,1},\dots,U_{i,i}$. Since each of these neighborhoods are disjoint $P_{n_{i}}(z)$ has $i$ distinct roots in $U_{1}$. The euclidian Mahler measure of $P_{n_{i}}(z)$ is then bounded below by $i*m_{1}$. We have constructed a subsequence of $\{P_{n}(z)\}$ whose euclidian Mahler measure diverges. The theorem follows.
 \end{proof}

 Next we review the methods developed in \cite{biggs} and extended in \cite{biggs-2} for finding 
 equi-modular curves for families of polynomials. In \cite{biggs} and \cite{biggs-2} the general case is considered.
 \newline
 
 Let $\lambda_{1}$ and $\lambda_{2}$ be polynomials with complex coefficients. Our goal is to find some description of the curve in the complex plane where they have the same modulus, $|\lambda_{1}(z)|=|\lambda_{2}(z)|$. If $z_{0}$ is such a point then $\lambda_{1}(z_{0})=s \lambda_{2}(z_{0})$ where $|s|=1$. Consider the matrix 
 
 $$T(z)=\begin{pmatrix}\lambda_{1}(z)&0\\0&\lambda_{2}(z)\end{pmatrix}.$$

 Then we see that $\lambda_{2}(z_{0})$ and $s\lambda_{2}(z_{0})$ are both roots of the characteristic equation of $T(z)$, 
 
 $$p(z)=det(\lambda I -T(z))=\lambda^{2}-(\lambda_{1}(z)+\lambda_{2}(z))\lambda+\lambda_{1}(z)\lambda_{2}(z).$$
 
 Consider $s$ as a formal variable and $\lambda_{2}(z)$ as a simultaneous root of $p(z)$ and $p(sz)$. We can investigate such simultaneous zeros by finding when the determinate of Sylvester's matrix for these two polynomials vanishes. That is find $z$ such that 
 
 $$det(\begin{pmatrix}
 s^{2}&s(-\lambda_{1}(z)-\lambda_{2}(z))&\lambda_{1}(z)\lambda_{2}(z)&0\\
 0&s^{2}&s^{2}(-\lambda_{1}(z)-\lambda_{2}(z))&\lambda_{1}(z)\lambda_{2}(z)&0\\
 1&1(-\lambda_{1}(z)-\lambda_{2}(z))&\lambda_{1}(z)\lambda_{2}(z)&0\\
 0&1&1(-\lambda_{1}(z)-\lambda_{2}(z))&\lambda_{1}(z)\lambda_{2}(z)\end{pmatrix}=0.$$
 
 Subtracting row 3 from row 1 and row 4 from row 2 one finds we should consider $z$ which are zeros of
 
 \begin{eqnarray}& &det(\begin{pmatrix}
  s^{2}-1&(s-1)(-\lambda_{1}(z)-\lambda_{2}(z))&0&0\\
  0& s^{2}-1&(s-1)(-\lambda_{1}(z)-\lambda_{2}(z))&0\\
  1&-\lambda_{1}(z)-\lambda_{2}(z)&\lambda_{1}(z)\lambda_{2}(z)&0\\
  0&1&-\lambda_{1}(z)-\lambda_{2}(z)&\lambda_{1}(z)\lambda_{2}(z)
  \end{pmatrix}) \nonumber \\ 
  &=&s (s-1)^{2} \lambda_{1}(z)\lambda_{2}(z)[-\lambda_{1}^{2}(z)-\lambda_{1}(z)\lambda_{2}(z)-\lambda_{2}^{2}(z)+t\lambda_{1}(z)\lambda_{2}(z)] \nonumber 
  \end{eqnarray}
  
  where $t=s+s^{-1}+2$. If we return to the situation where s is a unit modulus complex number, $s=cos(\theta)+isin(\theta)$, then $t=2 cos(\theta)+2=2(cos(2 \frac{\theta}{2})+1)=4 cos^{2}(\frac{\theta}{2})$. We exclude from consideration those points where one of $\lambda_{1}(z)$ or $\lambda_{2}(z)$ is zero. We have shown the following theorem holds. 
  
  \begin{thm} \label{V} A point $z_{0}$ is a point of equal modulus for $\lambda_{1}$ and  $\lambda_{2}$ and thus a limit point for the zeros of the family of polynomials $P_{n}(z)=a_{1}(z)\  \lambda_{1}(z)^{n}+a_{2}\ \lambda_{2}(z)^{n}$ if and only if for
  
  $$ v(t,z_{0})=-\lambda_{1}^{2}(z_{0})-\lambda_{1}(z_{0})\lambda_{2}(z_{0})-\lambda_{2}^{2}(z_{0})+t\lambda_{1}(z_{0})\lambda_{2}(z_{0})=0$$
  
  for some $t\in [0,4]$.   \end{thm}
  
  \begin{rmk} \label{Jacobian} \end{rmk}
  Think of $v(t,z)$ as a map from $\mathbb{R} \times \mathbb{R}^{2}$ to $\mathbb{R}^{2}$. If $v(t^{*},z_{0})=\bar{0}$ and $0 <t^{*}<4$, we want to ensure that $z_{0}$ is not an isolated point. Apply the inverse function theorem and check to see if the Jacobian of $v(t,z)$ with respect to the second two coordinates evaluated at $(t^{*},z_{0})$ is nonsingular.
 Through careful application of this theorem we can determine whether or not the equi-modular curve corresponding to a repeated rational surgery has any non-isolated points outside the unit circle. Then Theorem \ref{zeros} allows us to conclude that Mahler measure of the Kauffman bracket for this family of links  diverges as we add more rational tangles. 
 
 \section{Examples}
 
We will be investigating the zero locus of the Kauffman bracket for links found by repeated surgeries on a unlinked pair of simple closed components as in Figure
  \ref{fig: n_suregery} . It should be noted that the non-isolated points depend only on the rational tangle and not on the link on which we perform our surgery. The link on which we perform our surgery only contributes isolated points to the zero locus of the family of Kauffman brackets.
 
 \begin{figure}[h]
\begin{center}
\includegraphics{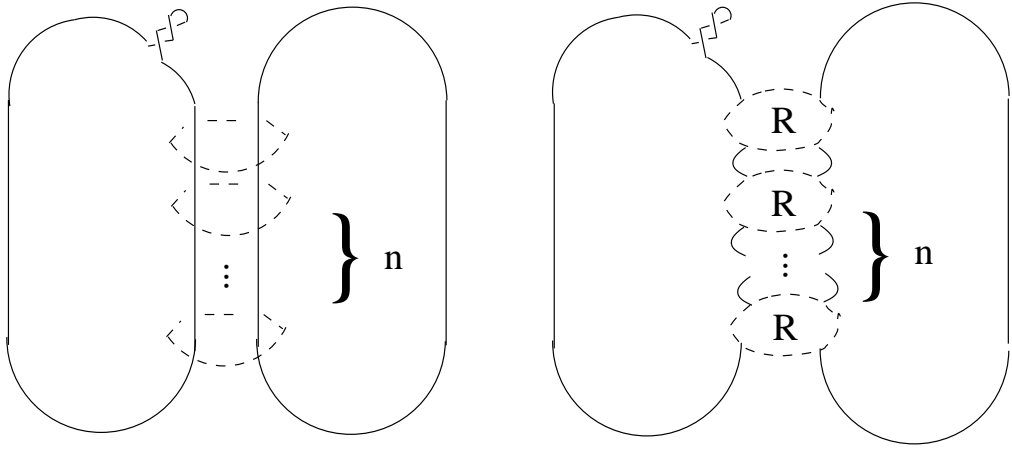}
\end{center}
\caption{repeated rational surgeries} \label{fig: n_suregery}
\end{figure}

 \begin{figure}[h]
\begin{center}
\includegraphics{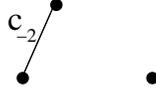}
\end{center}
\caption{The graph G} \label{fig: G}
\end{figure}

Using the notation of  Theorem \ref{KBform}, G is two vertices and no edges. Then $S_{1}=d$ and $S_{2}=0$. After performing n repeated rational surgeries we find that 

 $$W(G_{n})=d^{k<G_{n}>-|V(G_{n})|-1} [a_{1,1}^n (d-1)+(a_{1,1}+d^{2} a_{1,2})^{n}]$$
 
We can not immediately apply Theorem \ref{zeros} as our constituent functions may not be polynomials. However, we can utilize the Kauffman bracket by adding some nugatory crossings both to our two component unlink and to the rational tangle that we will repeatedly insert by surgery. To this end we will perform our rational surgeries on two component twisted unlink in Figure \ref{fig: n_suregery}. Then after each rational surgery perform some number of Reidemeister 1 moves. This process can be represented by repeated graph addition. The graph G to which we will perform the corresponding repeated graph addition is as in Figure \ref{fig: G}. 

It follows that $S_{1}=d^{5}A^{6}$ and $S_{2}=0$. Thus, neglecting powers of d, and making the substitutions as given in Theorem \ref{w2kb}

\begin{eqnarray} \nonumber
W(G_{n})&=&d^{k<G_{n}>-|V(G_{n})|} [a_{1,1}^n d^{5}*A^{6}(1-d^{-2})+(a_{1,1}+d^{2} a_{1,2})^{n}d^{3}A^{6}]\\ \nonumber
&=&d^{\Gamma} (a_{1,1}^n d^{2}A^{6}(1-d^{-2})+(a_{1,1}+d^{2} a_{1,2})^{n}A^{6})\\ \nonumber
&=&d^{\Gamma}(a_{1,1}^{n}(A^{10}+A^{6}+A^{2})+(a_{1,1}+a_{1,2}*d^{2})^{n} A^{6}) \\ \nonumber
&=&d^{\Gamma}A^{2}(a_{1,1}^{n}(A^{8}+A^{4}+1)+(a_{1,1}+a_{1,2}d^{2})^{n} A^{4})
\end{eqnarray}
 
\begin{Ex}\end{Ex}
 
 First we examine the case of performing a -1/2 surgery as in Example 1. There we find that $a_{1,1}=A^{-2}$ and $a_{1,1}+d^{2}a_{1,2}=(-1)^{n}A^{-3}$. We may expresses the twist link with 2n crossings as 
 
 $$<T_{2n}>=d^{\Gamma} (A^{-3n+2})(A^{n}(A^{8}+A^{4}+1)+A^{4})$$
 
 We may apply Theorem \ref{eqcrv} by excluding the origin and the roots of $d$ from our domain D.
 
Without the aid of Theorem \ref{V} we can see that the points where $A^{n}$ and 1 have equal modulus is the unit circle, and thus that is where the zeros of the family converge. This has been shown several times before including in \cite{zeros} and \cite{Twisting}.

\begin{Ex} \end{Ex}

Consider adding, by rational surgery, the  rational tangle $[-2,-1]$ and then performing several Reidemeister 1 moves.  This inserts a tangle as in Figure \ref{fig: shift_2_1}. The complete process can be viewed as  adding the graph in Figure \ref{fig: shift_2_1} to the appropriate pair of vertices in the graph G. Iterating this process then gives a family of links $\{T_{(n,[-2.-1])}\}$ where n is the number of graph additions performed. 

\begin{figure}[h]
\begin{center}
\includegraphics{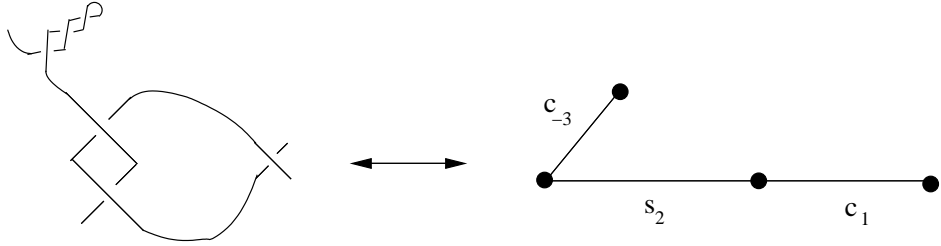}
\end{center}
\caption{The rational tangle and its graph} \label{fig: shift_2_1}
\end{figure}

By Theorem \ref{KBform} the Kauffman bracket is,

$$<T_{n,[-2,-1]}>=d^{\Gamma_{n}}A^{\Delta_{n}}
((A^{8}-A^{4}+1)^{n}(A^{8}+A^{4}+1)+(A^{12}-A^{8} -1)^{n}A^{4})$$

Interpreting this according to Theorem \ref{zeros} let 

\begin{eqnarray}a_{1}&=&A^{8}+A^{4}+1 \\ \nonumber
\lambda_{1}&=&A^{8}-A^{4}+1 \\ \nonumber
a_{2}&=&A^{4}\\ \nonumber
\lambda_{2}&=&A^{12}-A^{8} -1\nonumber
\end{eqnarray}

Then the family of zeros, $\{ \mathcal{Z}(<T_{n,[-2,-1]}>)\}$,  converges to some isolated points and the equi-modular curves defined by $|\lambda_{1}|=|\lambda_{2}|$. 

Theorem \ref{V} allows us to find points $z_{0}$  on the equi-modular curve by finding roots of 
$$v(t,A)=-(A^{12}-A^{4})^{2}+t(A^{8}-A^{4}+1)(A^{12}-A^{8} -1)$$
when $0 \leq t \leq4$. 

Thus the roots of $v(0,A)$, $v(4,A)$ and $v(5/4,A)$ are points on the equi-modular curve to which the zeros of the Kauffman bracket of this family of links converges. 

The zeros of $<T_{n,[-2,-1]}>$  for $n=10$ and $n=20$ are plotted  in Figure \ref{fig: crv_2-1}, along with the roots of $v(0,A)$, $v(4,A)$ and $v(5/4,A)$. Roots of $<T_{n,[-2,-1]}>$ are represented by open circles. Roots of $v(0,A)$, $v(4,A)$ and $v(5/4,A)$ are represented by a black square, a black diamond, and a solid dot respectively. One observes that the roots of $v(5/4,A)$ fall outside the unit circle. One can also show that these roots are not isolated points as noted in Remark \ref{Jacobian}. Thus by Theorem \ref{zeros} the sequence $\{M( <T_{n,[-2,-1]}>)\}$, the Mahler measure of the Kauffman brackets for links in this family, diverges!

\begin{Ex}\end{Ex}

The same analysis can be done for the rational surgeries whose corresponding shifted tangle is given in Figures \ref{fig: shift_2_2}, \ref{fig: 3_2}, \ref{fig: 3_3}, \ref{fig: 2_2_2}, and \ref{fig: 3_2_2}. The plots of zeros are  Figures \ref{fig: crv_2-2}, \ref{fig: crv_3-2}, \ref{fig: crv_3-3}, \ref{fig: crv_2-2-2}, and \ref{fig: crv_3-2-2} respectively. In each case it is a simple exercise to verify that there is a non-isolated point (in particular  of the zeros of $v(5/4,A)$) of the equi-modular curve. As such, for each of these families of links the Mahler measure diverges. Let $<T_{n,[a_{1}, \dots a_{m}]}>$ denote the Kauffman bracket of the link found by performing n rational $[a_{1}, \dots a_{m}]$ surgeries on  two unlinked components (shifted as in Figure \ref{fig: G}).
 
 \section{Twist number and Volume}
We would like to compare the behavior of Mahler measure and the hyperbolic volume of a link compliment under the repeated surgeries as described above. There are several results relating properties of link diagrams to the hyperbolic volume of its compliment. 
Consider a link $L$ in $S^{3}$. Say that a diagram $D$ of $L$ is \textit{twist reduced} if  any two twist equivalent crossings are connected by a chain of bi-gons. For a twist reduced diagram the number of edges in its corresponding graph as described in section \ref{wPolyTwistpoly} equals the twist number of the diagram.  We call a  diagram  A-adequate if the number of closed components found by smoothing each crossing positively is greater than the number of components found by smoothing all but one crossing positively and one negatively. Similarly we call a diagram B-adequate if the same statement holds, switching the roles of positive and negative. Call a diagram adequate if it is both A- and B-adequate. Futer, Kalfagianni and Purcell prove the following theorem in \cite{effi}.

 \begin{thm}
 Let K be a link in $S^{3}$ with a prime, twist reduced diagram D. Assume that the the twist number of D is greater than or equal to 2, and that every twist equivalence class has at least 3 crossings. Then
 $$\alpha(\frac{1}{3} tw(D)-1 )< vol(S^{3}\backslash K)$$
 
 for a constant $\alpha$.
  \end{thm}

Notice that if we repeatedly perform a $[-3,-3]$ rational surgery ( repeatedly adding the rational tangle in Figure \ref{fig: 3_3} with the nugatory crossings removed)  on a pair of unlinked components then each member of this family is prime, twist reduced and adequate (as the diagram is alternating). By the above theorem the volume diverges as the twist number diverges. Since the Mahler measure also diverges we can continue the analogy between their behaviors.

\section{Conclusion and Future Directions}

  In \cite{Twisting}, \cite{cyc}, \cite{silver} and \cite{silver_williams} conditions under which Mahler measure converges was investigated. First we were able to refine the geometric bound for the Mahler measure using the W-polynomial. Then we consider when Mahler measure diverges and developed techniques for studying the question.  One can think of these equi-modular curves as "belonging" to the associated rational tangle since the link on which we perform such surgeries does not contribute to those curves. Thorough investigation of such curves would be interesting. Also,  to applying these techniques to issues of convergence as in Example 3 would be interesting. If in fact it is true that convergence of the family of zeros to some isolated points and to the unit circle implies the convergence of the Mahler measure then understanding the isolated points may give some geometric insight into contributions to the Mahler measure. We hope to perform these investigations in the future. 

\begin{figure}[h!]
\begin{center}
\includegraphics{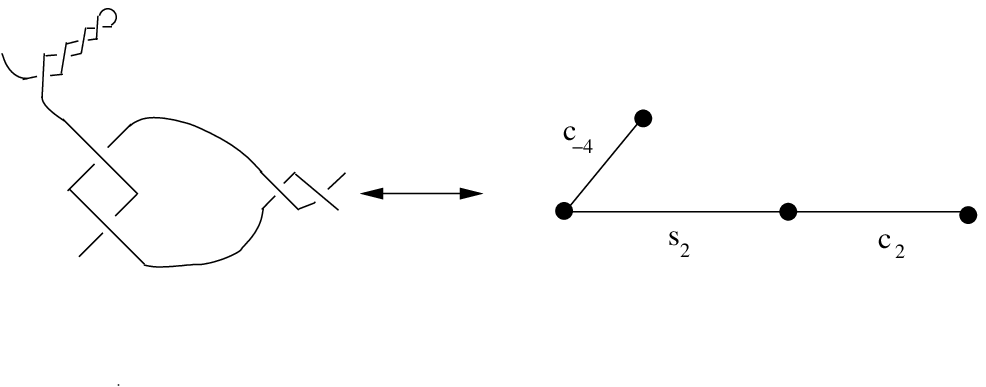}
\end{center}
\caption{Rational tangle [-2,-2] shifted} \label{fig: shift_2_2}
\end{figure}

\begin{figure}[h!]
\begin{center}
\includegraphics{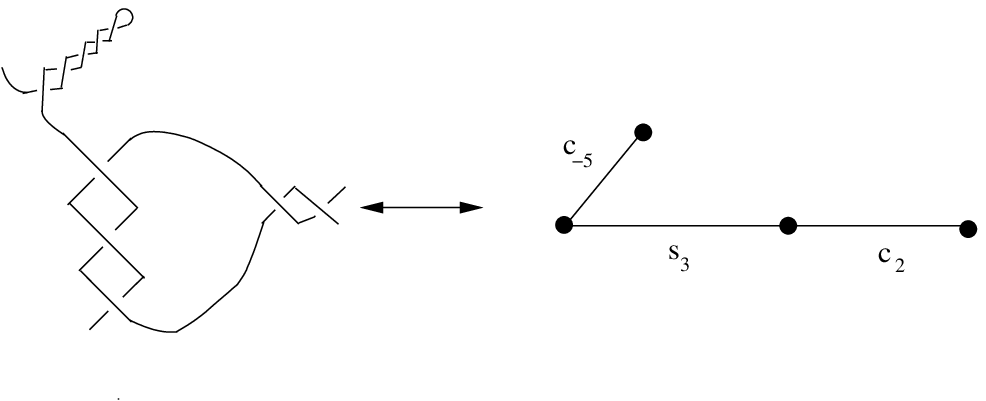}
\end{center}
\caption{Rational tangle [-3,-2] shifted} \label{fig: 3_2}
\end{figure}

\begin{figure}[h!] 
\begin{center}
\includegraphics{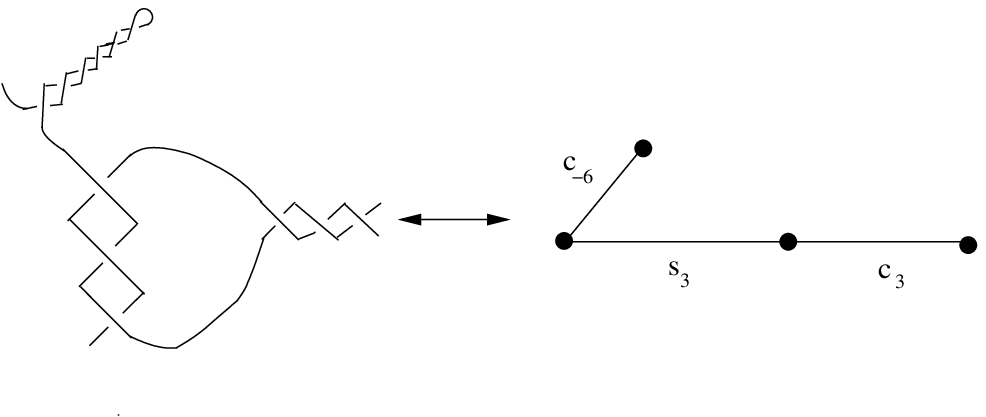}
\end{center}
\caption{Rational tangle [-3,-3] shifted} \label{fig: 3_3}
\end{figure}

\pagebreak 

\begin{figure}[h!]
\begin{center}
\includegraphics{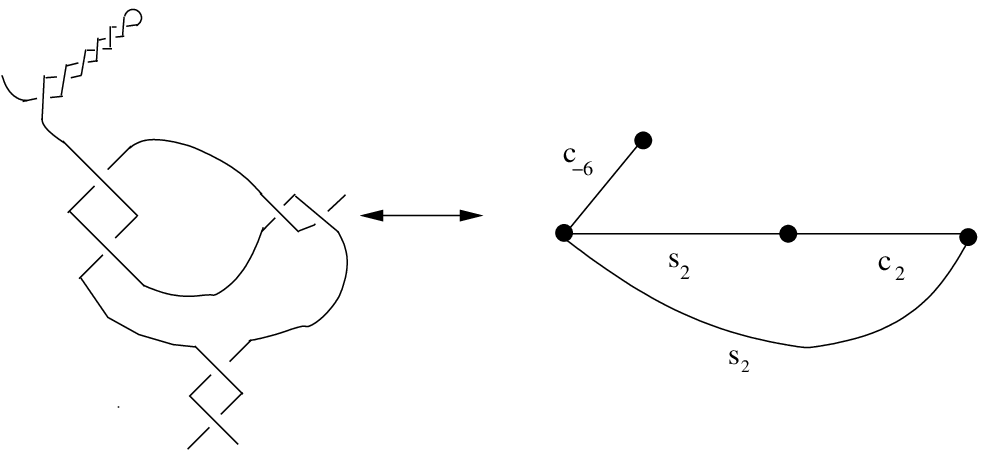}
\end{center}
\caption{Rational tangle [-2,-2.-2] shifted} \label{fig: 2_2_2}
\end{figure}

\begin{figure}[h!]
\begin{center}
\includegraphics{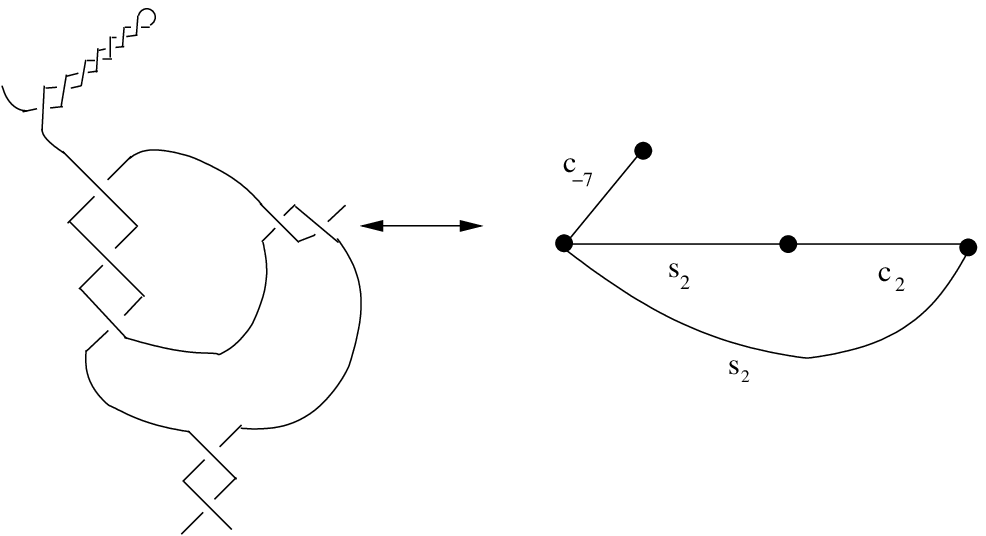}
\end{center}
\caption{Rational tangle [-3,-2,-3] shifted} \label{fig: 3_2_2}
\end{figure}

\pagebreak

\begin{figure}[h]
\hspace*{-.7in}\includegraphics{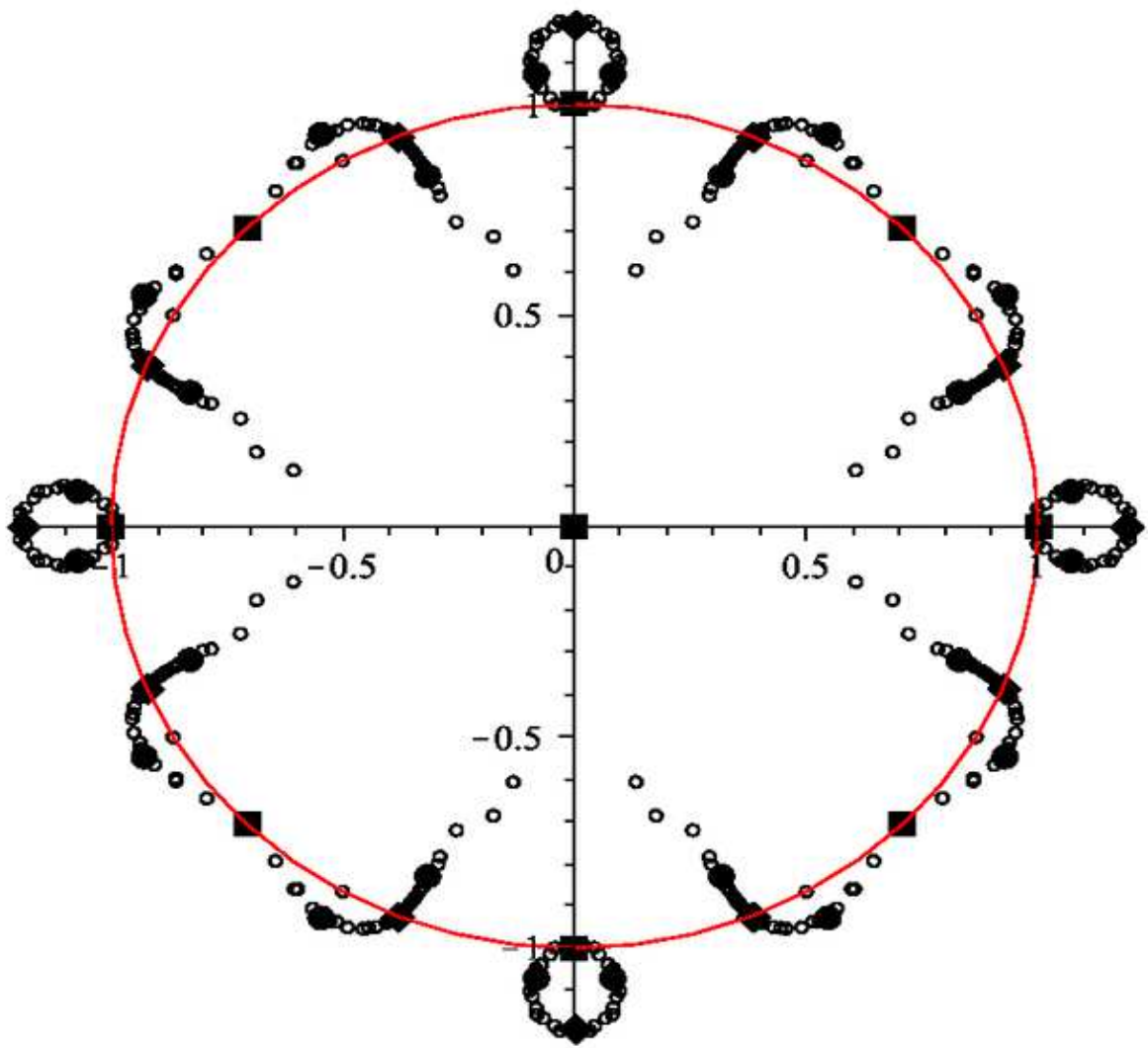}
\vspace*{-1.5in}
\caption{$v(5/4,black\ dot)=0$,$v(0,black\ square)=0$ and $v(4,black\ diamond)=0$. Roots of $<T_{10,[-2,-1]}>=0$ and $<T_{20,[-2,-1]}>=0$ are open circles. } \label{fig: crv_2-1}
\end{figure}

\begin{figure}[h]
\hspace*{-.7in} \includegraphics{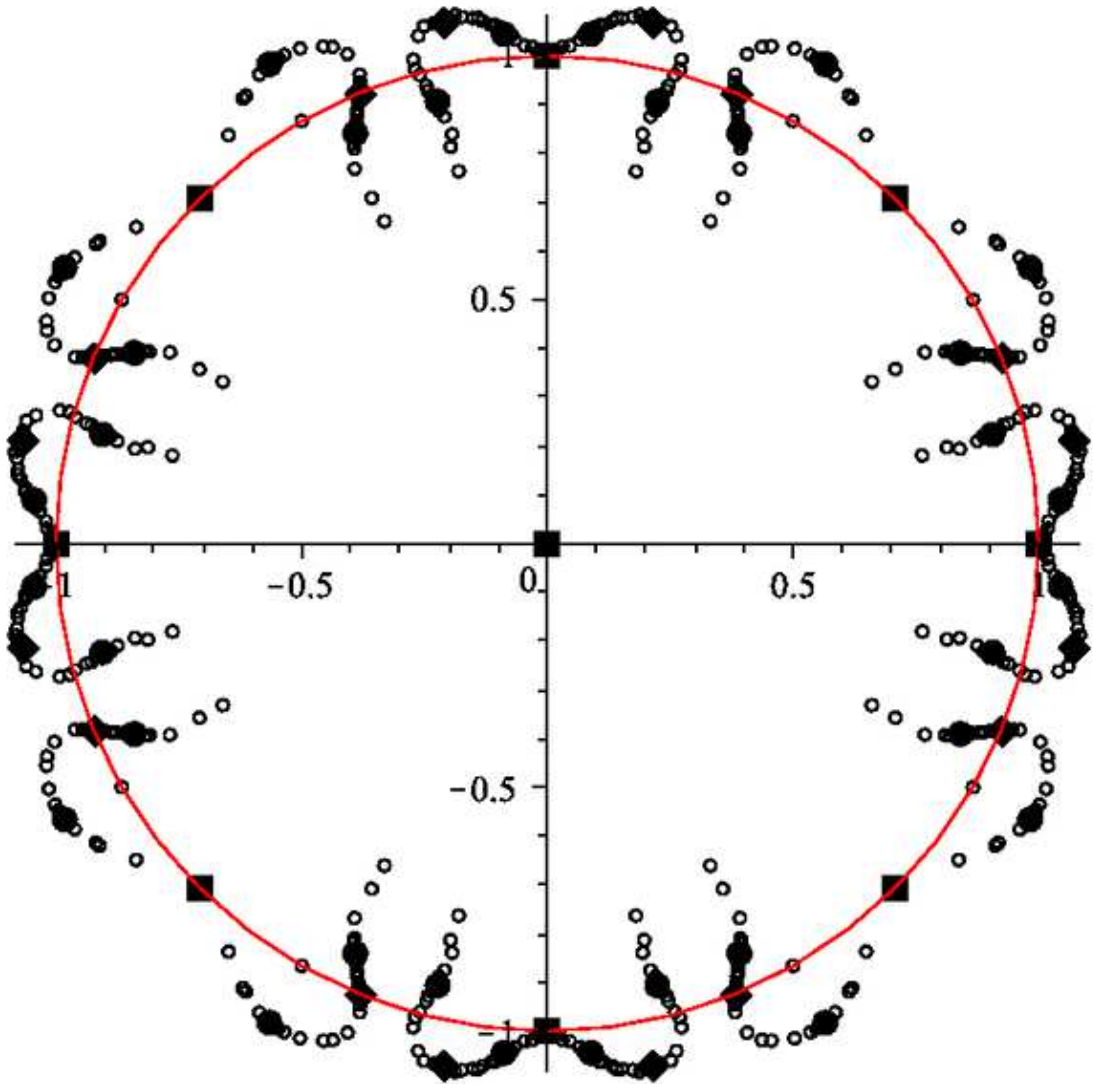}
\vspace*{-1.5in}
\caption{$v(5/4,black\ dot)=0$,$v(0,black\ square)=0$ and $v(4,black\ diamond)=0$. Roots of $<T_{10,[-2,-2]}>=0$ and $<T_{20,[-2,-2]}>=0$ are open circles. } \label{fig: crv_2-2}
\end{figure}

\begin{figure}[h]
\hspace*{-.7in} \includegraphics{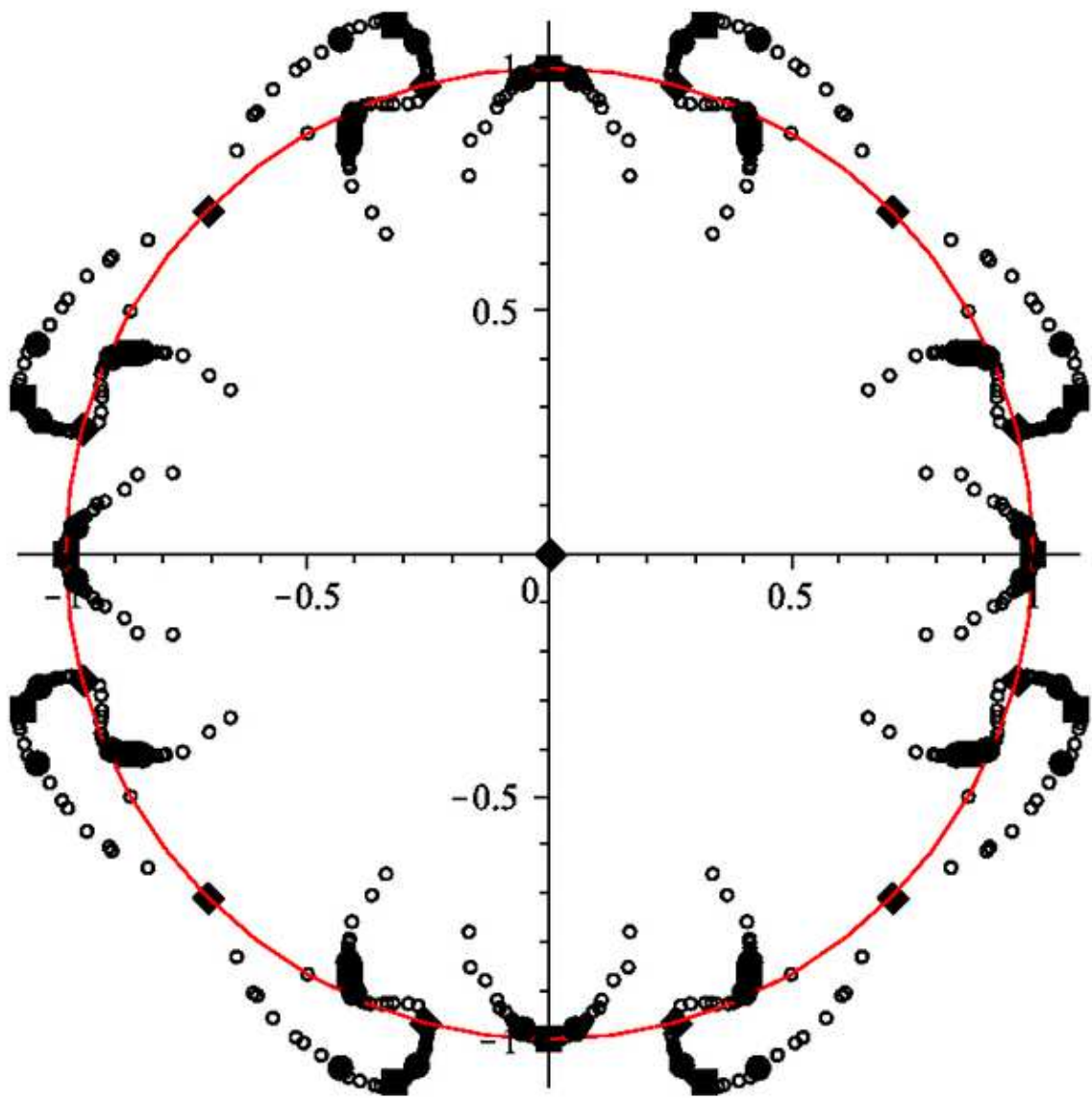}
\vspace*{-1.5in}
\caption{$v(5/4,black\ dot)=0$,$v(0,black\ square)=0$ and $v(4,black\ diamond)=0$. Roots of 
$<T_{10,[-3,-2]}>=0$ and $<T_{20,[-3,-2]}>=0$ are open circles.} \label{fig: crv_3-2}
\end{figure}

\begin{figure}[h]
\hspace*{-.5in} \includegraphics{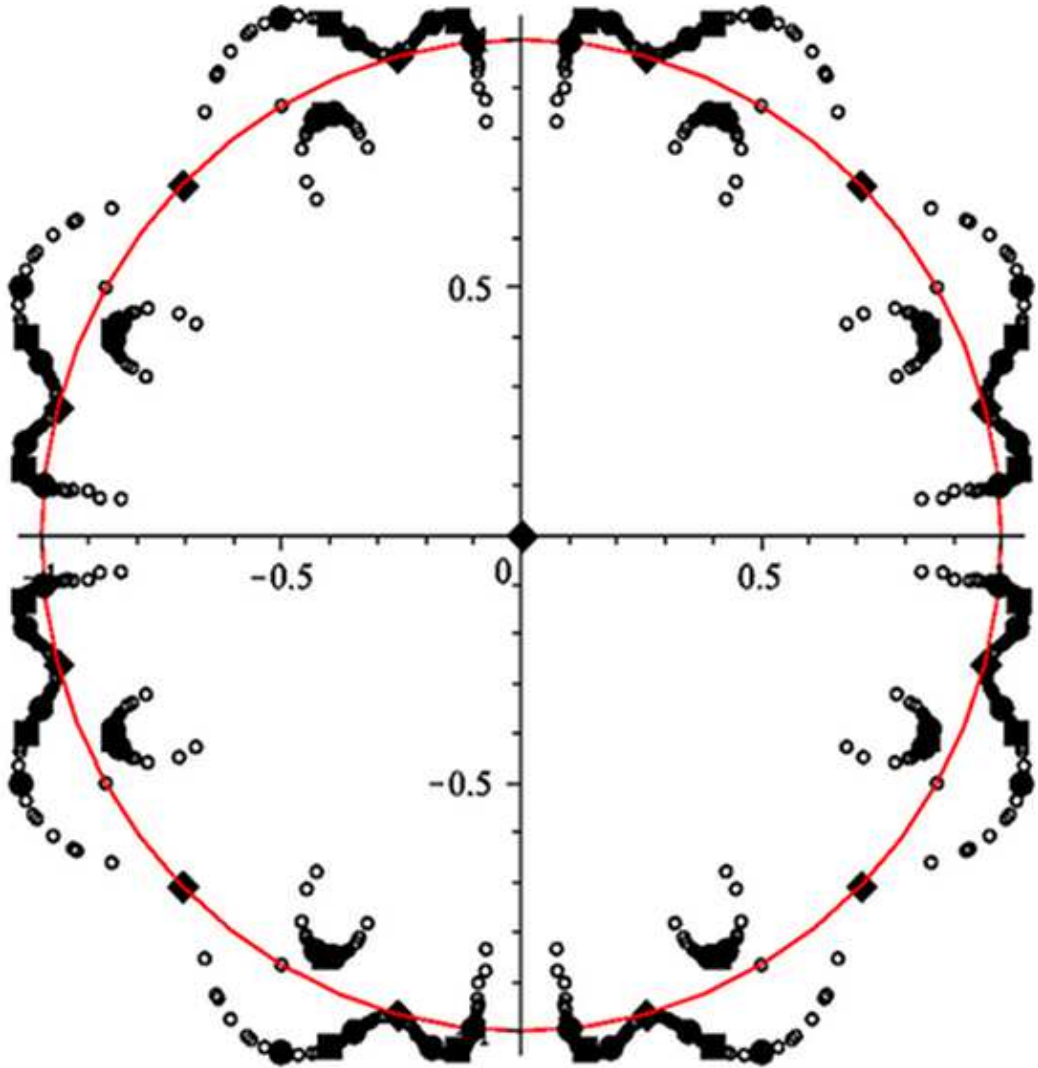}
\vspace*{-1.5in}
\caption{$v(5/4,black\ dot)=0$,$v(0,black\ square)=0$ and $v(4,black\ diamond)=0$. Roots of $<T_{10,[-3,-3]}>=0$ and $<T_{20,[-3,-3]}>=0$ are open circles. } \label{fig: crv_3-3}
\end{figure}

\begin{figure}[h]
\hspace*{-.7in} \includegraphics{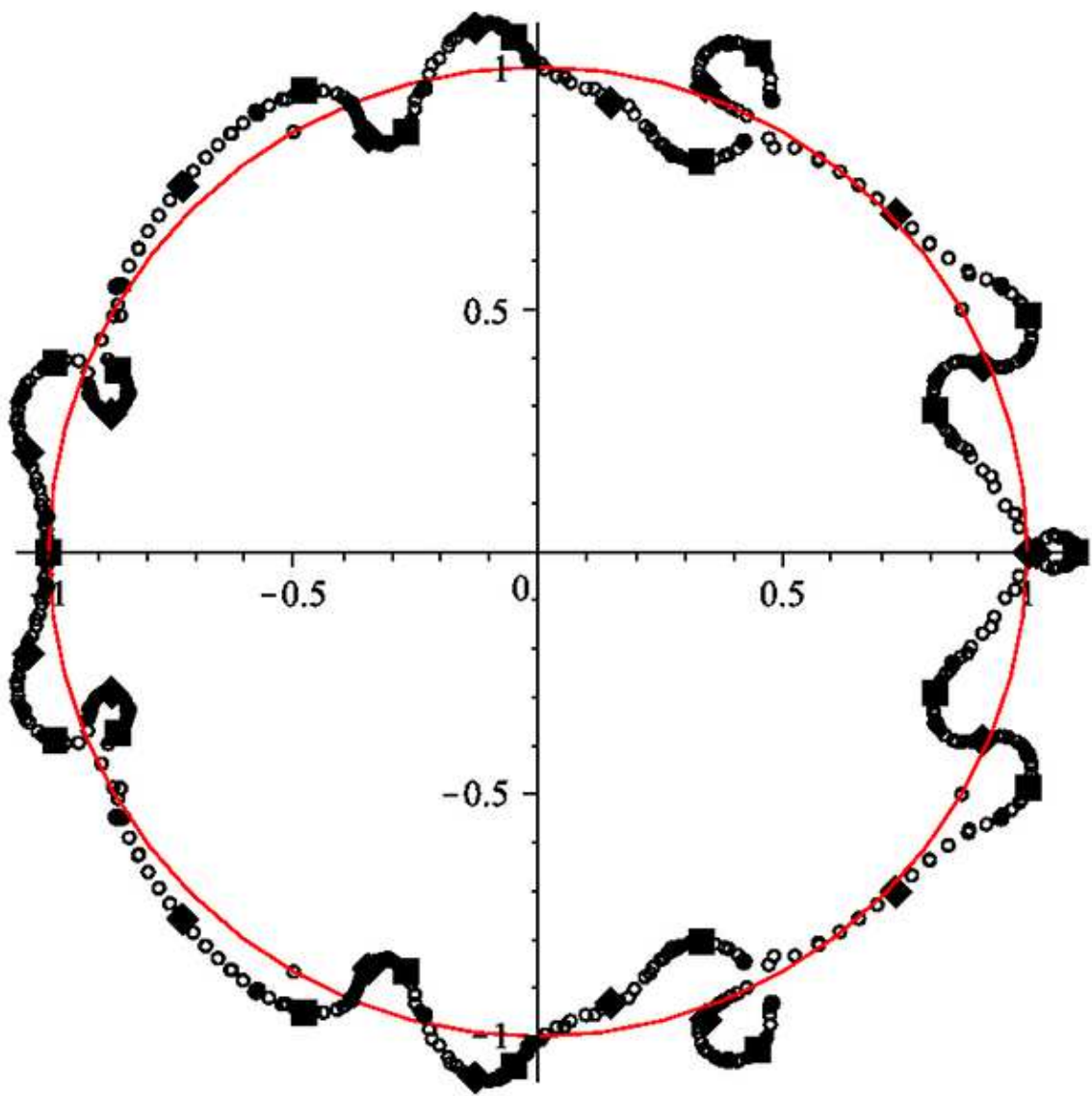}
\vspace*{-1.5in}
\caption{$v(5/4,black\ dot)=0$,$v(0,black\ square)=0$ and $v(4,black\ diamond)=0$. Roots of 
$<T_{10,[-2,-2,-2]}>=0$ and $<T_{20,[-2,-2,-2]}>=0$ are open circles} \label{fig: crv_2-2-2}
\end{figure}

\begin{figure}[h]
\hspace*{-.7in} \includegraphics{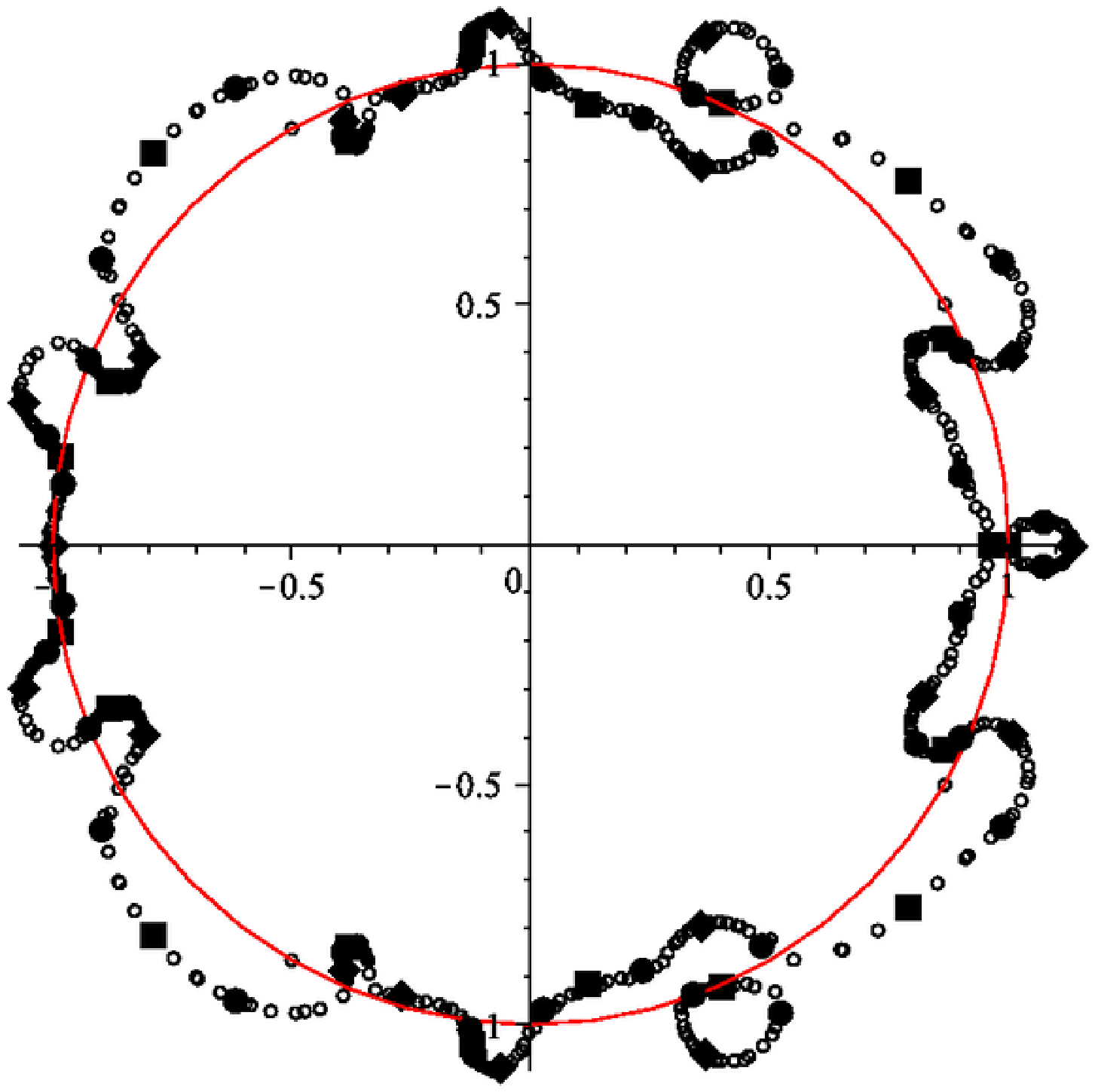}
\vspace*{-1.5in}
\caption{$v(5/4,black\ dot)=0$,$v(0,black\ square)=0$ and $v(4,black\ diamond)=0$. Roots of 
$<T_{10,[-3,-2,-2]}>=0$ and $<T_{20,[-3,-2,-2]}>=0$are open circles} \label{fig: crv_3-2-2}
\end{figure}


\begin{thebibliography}{0000}
\bibitem{beraha_kahane_weiss}S. Beraha, J. Kahane,N.J. Wiess, \textit {Limits of zeros  of recursively defined families of polynomials} in Studies in Foundations and Combinatorics (Advances in Mathematics Supplementary Studies, Vol. 1), ed. G.-C. 
Rota (Academic Press, New York, 1978). 
\bibitem{biggs} N. Biggs, \textit{ Equi-Modular cures} CDAM Research Report Series, LSE-CDAM-2000-17, September 2000.
\bibitem{biggs-2} N. Biggs, \textit{Equi-modular curves for reducible matrices}, CDAM Research Report, LSE-CDAM-2001-01, January 2001
\bibitem{tutte}B. Bollobas, O. Riordan, A A \textit{Tutte Polynomial for Colored Graphs} ,Combinatorics,Probability and Computing 8 (1999) 45-93
\bibitem{Twisting}A. Champanerkar, I. Kofman,\textit{On the Mahler Measure of Jones Polynomials Under Twisting}, Algebraic and Geometric Topology 5 (2005) 1-22 
\bibitem{cyc} A. Champanerkar, I. Kofman, \textit{On Links with Cyclotomic Jones Polynomials}, Algebraic and Geometric Topology 6
              (2006) 1655Ð1668
\bibitem{chang_schrock} S.C. Chang,  R. Shrock, \textit{Zeros of Jones polynomials for families of knots and links}, Physica A: Statistical Mechanics and its Applications
Volume 301, Issues 1-4, 1 December 2001, Pages 196-218 
\bibitem{effi} D. Futer, E Kalfagianni, J.S. Purcell \textit{Dehn Filling, Volume, and the Jones Polynomial}, J. Diff. Geom. Volume 78, Number 3 (2008), 429-464. 
\bibitem{twist} M. Lackenby, \textit{Volume of Hyperbolic Alternating Link Complements}, Proceedings of the London Mathematical Society (2004), 88:1:204-224
\bibitem{mm}A. Schinzel, \textit{The Mahler measure of polynomials}, from "Number theory and its applications (Ankara, 1996)", Lecture Notes in Pure and Applied Mathematics 204, Dekker, New York (1999) 171-183, New York, MR!661667
\bibitem{mainpap} X. Jin, F. Zhang, \textit{The Replacements of signed graphs and Kauffman brackets of links},
Advances in Applied Mathematics
Volume 39, Issue 2, August 2007, Pages 155-172 
\bibitem{eqlnks}X. Jin, F. Zhang, \textit{The Kauffman brackets for equivalence classes of links}, Advances in Applied Mathematics
Volume 34, Issue 1, January 2005, Pages 47-64 
\bibitem{silver} D. Silver, A. Stoimenow, S. Williams, \textit{Euclidean Mahler measure and twisted links}, Algebraic and Geometric Topology 6 (2006) 581-602
\bibitem{silver_williams} D. Silver, S. Williams, \textit{Mahler Measure of the Alexander Polynomial} , Journal of the London Mathematical Society,69, 3, 767-782.
\bibitem{sokal_1} A.D. Sokal,\textit{Chromatic Roots are Dense in the Whole Complex Plane}, Combin. Probab. Comput. 13, 221 (2004), cond-mat/0012369. 
\bibitem{sokal_2}A.D. Sokal, \textit{Transfer Matrices and Partition-Function Zeros for Antiferromagnetic Potts Models}, Journal of statistical physics, 	Volume 135, Number 2 , April, 2009,279-373
\bibitem{zeros} F.Y. Wu, J. Wang, \textit{Zeroes of the Jones polynomial} ,Physica A  296, Number 3,  2001  483-494(12)

\end{thebibliography}
\end{document}